\documentclass[letter,12pt]{article}

\usepackage[english]{babel}
\usepackage[utf8]{inputenc}

\usepackage[margin=1.2 in, top=1 in, bottom= 1.2 in]{geometry}

\makeatletter
\g@addto@macro\normalsize{%
  \setlength\abovedisplayskip{7pt}
  \setlength\belowdisplayskip{7pt}
  \setlength\abovedisplayshortskip{7pt}
  \setlength\belowdisplayshortskip{7pt}
}
\makeatother

\usepackage{tocloft}

\usepackage{amsfonts}
\usepackage{mathrsfs}
\usepackage{bbm}
\usepackage{latexsym}
\usepackage{math dots}
\usepackage{amssymb}
\usepackage{mathtools}
\usepackage{relsize}
\usepackage{lipsum}

\usepackage{enumitem}
\setlist{nolistsep} 	
\usepackage{amsthm}

\usepackage{xcolor}
\definecolor{Color1}{rgb}{0.0, 0.42, 0.47}
\definecolor{Color2}{rgb}{0.78, 0.11, 0.0}

\usepackage{titlesec}
\titleformat{\section}
  {\large\center\bfseries}
  {\thesection.}{.7em}{}
\titlespacing*{\section}{0pt}{3.5ex plus 0ex minus 0ex}{1.5ex plus 0ex}
\titleformat{\subsection}
  {\center\bfseries}
  {\thesubsection.}{.7em}{}
\titlespacing*{\subsection}{0pt}{3.5ex plus 0ex minus 0ex}{1.5ex plus 0ex}
\titleformat{\subsubsection}
  {\center\bfseries}
  {\thesubsubsection.}{.7em}{}
\titlespacing*{\subsubsection}{0pt}{3.5ex plus 0ex minus 0ex}{1.5ex plus 0ex}

\addto\captionsenglish{}

\usepackage{titling}
\makeatletter
\renewenvironment{abstract}{
\begin{center}
{\bfseries \large\abstractname\vspace{\z@}}
\end{center}
\quotation
}
\makeatother

\makeatletter
\newenvironment{acknowledgements}{
\begin{center}
{\bfseries \large Acknowledgments\vspace{\z@}}
\end{center}
}

\usepackage[linktocpage=true]{hyperref}
\usepackage[capitalize]{cleveref}
\hypersetup{citecolor = Color2,colorlinks,
			linkcolor = black,
			urlcolor = Color2}

\newtheoremstyle{plain}{3mm}{3mm}{\slshape}{}{\bfseries}{.}{.5em}{}
\newtheoremstyle{definition}{2mm}{2mm}{}{}{\bfseries}{.}{.5em}{}
\theoremstyle{plain}
	
\newtheorem{Theorem}{Theorem}

\newtheorem{Lemma}[Theorem]{Lemma}
\newtheorem{Proposition}[Theorem]{Proposition}
\newtheorem{Corollary}[Theorem]{Corollary}

\theoremstyle{definition}
\newtheorem{Definition}[Theorem]{Definition}
\newtheorem{Remark}[Theorem]{Remark}

\theoremstyle{plain} 
\newcounter{MainTheoremCounter}

\theoremstyle{plain}
\newtheorem*{namedthm}{\namedthmname}
\newcounter{namedthm}
	

\usepackage{chngcntr}
\counterwithin{Theorem}{section}

\numberwithin{equation}{section}

\allowdisplaybreaks

\newcommand{\Erdos}{Erd\H{o}s}

\newcommand{\Turan}{Tur{\'a}n}

\newcommand{\N}{\mathbb{N}}
\newcommand{\Z}{\mathbb{Z}}
\newcommand{\R}{\mathbb{R}}

\newcommand{\Q}{\mathbb{Q}}

\newcommand{\define}[1]{{\itshape #1}}

\renewcommand{\epsilon}{\varepsilon}
\renewcommand{\leq}{\leqslant}
\renewcommand{\geq}{\geqslant}
\renewcommand{\setminus}{\backslash}

\renewcommand{\P}{\mathbb{P}}
\renewcommand{\subset}{\subseteq}

\renewcommand{\d}{~\mathrm{d}}

\usepackage[normalem]{ulem}

\usepackage{cancel}

\usepackage{xargs}                      
\usepackage[colorinlistoftodos,prependcaption]{todonotes}

\newcommandx{\unsure}[2][1=]{\todo[linecolor=red,backgroundcolor=red!25,bordercolor=red,#1]{#2}}
\newcommandx{\change}[2][1=]{\todo[linecolor=blue,backgroundcolor=blue!25,bordercolor=blue,#1]{#2}}
\newcommandx{\info}[2][1=]{\todo[linecolor=OliveGreen,backgroundcolor=OliveGreen!25,bordercolor=OliveGreen,#1]{#2}}
\newcommandx{\improvement}[2][1=]{\todo[linecolor=Plum,backgroundcolor=Plum!25,bordercolor=Plum,#1]{#2}}
\newcommandx{\thiswillnotshow}[2][1=]{\todo[disable,#1]{#2}}

\author{By {\scshape Sai Sanjeev Balakrishnan}, {\scshape Félix Houde},\\ {\scshape Vahagn Hovhannisyan}, {\scshape Maryna Manskova} and {\scshape Yiqing Wang}}
\date{\small \today}
\title{\bfseries Arbitrarily long strings of consecutive primes in special sets}

\begin{document}

\maketitle

\begin{abstract}
    Let $F(x)$ be a function of the form $ \sum_{i=1}^r d_i x^{\rho_i}$ where $d_1,\ldots,d_r\in\R$, $0 \leq \rho_1 < \ldots < \rho_r,$ $\rho_r \not\in \Z,\rho_i \in \mathbb{R}$ for $ 1 \leq i \leq r$ and $d_r\not=0$. We prove that sets of the form $\{ n \in \N: \{ F(n) \} \in U \}$ for any non-empty open set $U \subset [0,1)$ contain arbitrarily long strings of consecutive primes.
\end{abstract}
\renewcommand{\thefootnote}{}
\renewcommand{\thefootnote}{\arabic{footnote}}
\tableofcontents
\thispagestyle{empty}

\section{Introduction}

A natural question that arises in the study of the distribution of prime numbers is determining which sets $\mathcal{A} \subseteq \N$ contain arbitrarily long strings of consecutive primes $p_{k+1},p_{k+2},\dots,p_{k+m}$, where $p_n$ is the $n$th prime number in $\mathbb{N}$. When $\mathcal{A}$ is an arithmetic progression of the form $a\N+b$ with $\gcd(a,b)=1$, Shiu famously showed that $\mathcal{A}$ contains arbitrarily long strings of consecutive primes and gave a lower bound on the number of such strings in $(a\N+b) \cap [1,x)$ (see Theorem 1 and 2 of \cite{shiu2000strings}). 

The recent breakthroughs of Zhang \cite{zhang2014bounded} and Maynard (\cite{maynard2015gaps} and \cite{Maynard_clusters}) on small gaps between primes introduced new approaches to investigate these questions, that yielded a strengthening of Shiu's result in \cite[Theorem 3.3]{Maynard_clusters}. Using these methods, Shiu's result was further strengthened in \cite[Corollary 1.2]{Radomskii_2021}, where it was shown that there exist arbitrarily long strings of primes in $(a\N+b)$, such that the difference of the largest and the smallest prime in the string is bounded by a fixed number. Another interesting application of the method yielded the existence of arbitrarily long strings of consecutive primes having a fixed integer $g \neq -1$, not a square, as a primitive root modulo those primes (\cite[Theorem 4.1]{Pollack_2014}), conditional on the generalized Riemann hypothesis.

The existence of arbitrarily long strings of consecutive primes has also been investigated when $\mathcal{A}$ is a Beatty sequence satisfying certain conditions. For an irrational number $\alpha > 0$, the associated Beatty sequence is the set of integers
\begin{align*}
    \mathcal{A} = \{ \lfloor n \alpha \rfloor : n \in \N \},
\end{align*}
where $\lfloor x \rfloor$ is the largest integer less than or equal to $x$. Baker and Zhao showed that the sequence of primes in any Beatty sequence associated with an irrational $\alpha$ has bounded gaps by using a variant of Maynard's method (\cite[Theorem 2]{RogerC2016}). Similarly, in \cite[Theorem 1.1]{chua2015bounded}, the same result was shown when $\alpha$ is of finite type, and, in addition, the proof implies that Beatty sequences for such $\alpha$ contain arbitrarily long strings of consecutive primes (see the remark after the proof of \cite[Theorem 1.1]{chua2015bounded}).
Yet another interesting result along these lines is \cite[Corollary 1.2]{Monks_2013}, which shows that when $\alpha$ is of finite type, the consecutive primes can be chosen such that they all lie in a single arithmetic progression. Surprisingly, a proof that for all irrational $\alpha > 0$ the set $\{ \lfloor n \alpha \rfloor: n \in \N \}$ contains arbitrarily long strings of consecutive primes has not appeared explicitly in the literature. However, as we will discuss next, we believe an even more general statement can be derived by modifying the recent results of Shao and Teräväinen.

Observe that any Beatty sequence can be written as a Bohr set, i.e.,
\begin{align*}
    \{ \lfloor n \alpha \rfloor : n \in \N \} = \{ n \in \N : \{ \alpha^{-1}n \} \in (1-\alpha^{-1},1)  \},
\end{align*}
where we define $\{ y \} = y - \lfloor y \rfloor$ for a real number $y$. Therefore, a natural question is whether one can generalize results of the above kind to sets of the form
\begin{equation*}
    \mathcal{A} = \{ n \in \N : \{ F(n) \} \in U  \},
\end{equation*}
where $F$ belongs to a family of polynomials and $U$ is, say, an open subset of $[0,1)$.
Using a weighted variant of the Maynard sieve (see \cite[Theorem 6.2]{Matom_ki_2017}), Shao and Teräväinen showed that for every open set $U,$ the primes in the set $\mathcal{A}$ have bounded gaps whenever $F$ is a real polynomial with at least one irrational coefficient besides the constant term (see \cite[Theorem 2.3]{Shao2021Bombieri}). In this case, $\mathcal{A}$ belongs to the family of so-called Nil-Bohr sets. Their method shows that for every $m \geq 1$, there are infinitely many intervals of bounded length containing $m$ primes that lie inside the set $\mathcal{A}$, as noted in the footnote of \cite[page 46]{Shao2021Bombieri}. Although not mentioned explicitly in their paper, it seems that their method also shows that those primes can be taken to be consecutive. Indeed, under some additional hypotheses, Maynard's theorem \cite[Theorem 3.1]{Maynard_clusters} gives a lower bound on the number of strings of consecutive primes found in a set of integers, and we believe that this should also carry over to the weighted version \cite[Theorem 6.2]{Matom_ki_2017} used by Shao and Teräväinen.


Inspired by this, we prove a similar result for sets $\mathcal{A}$ of the form
\begin{equation*}
    \mathcal{A} = \{ n \in \N : \{ F(n) \} \in U  \},
\end{equation*}
where $U$ is an open set, and $F$ belongs to a family of polynomial-like functions that we refer to as \textit{polynomials with real exponents}.

\begin{Definition}
    A \define{polynomial with non-integer exponents}\footnote{In the literature one might find many other names for slightly different versions of these functions, \textit{e.g.} fractional polynomials, exponential polynomials, polynomials with real exponents, signomials, etc.} is a function of the form
    \begin{equation*}
        f(x) = \sum_{i=1}^r d_i x^{\rho_i}
    \end{equation*}
    where $r \geq 1$, $0 < \rho_1 < \ldots < \rho_r$, $\rho_i\not\in\mathbb Z, \rho_i \in \mathbb{R}$ and $d_i \in \R$ for $1 \leq i \leq r$. A \define{polynomial with real exponents} is a function $F(x)$ of the form $f(x)+P(x),$ where $P(x)$ is a real polynomial and $f(x)$ is a polynomial with non-integer exponents.
\end{Definition}

We now state our main result.
\begin{Theorem}\label{main_result}
    Let $F(x) = \sum_{i=1}^r d_i x^{\rho_i}$ with $0 \leq \rho_1 < \ldots < \rho_r$ be a polynomial with real exponents, with $\rho_r \not\in \Z$ and $d_r \neq 0$. Let $\mathcal{A} := \{ n \in \N : \{ F(n) \} \in U \}$ for some non-empty open set $U \subset [0,1)$. Then $\mathcal{A}$ contains arbitrarily long strings of consecutive primes. Moreover, if $p_1<p_2<p_3<\ldots$ is the increasing enumeration of the prime numbers, then, for every positive integer $m$, there exists a constant $K = K(F,m)$ such that there are infinitely many $n$ with $p_{n+1}, \ldots, p_{n+m} \in \mathcal{A}$ and $p_{n+m} - p_{n+1} \leq K$.
\end{Theorem}

By passing to a subset of $\mathcal{A}$ if necessary, it follows from the theorem that one can take $U$ to be any subset of $[0,1)$ with a nonempty interior. The restriction $\rho_r\not\in\mathbb Z$ is technical and is explained in Section \ref{lastsection}.

\begin{acknowledgements}
This work was done as part of the 2023 edition of the Young Researchers in Mathematics Program. We would like to thank the Bernoulli Center for Fundamental Studies for funding and hosting the program. We are very grateful to Florian Richter for creating this opportunity and guiding us throughout the process. We would also like to thank Nadia Kaiser and Carole Weissenberger, who helped organize our stay in Lausanne. We also want to thank the anonymous reviewer for their careful reading of the paper and their useful comments on the presentation.

\end{acknowledgements}

\subsection{Some notation}
In what follows, we denote by $\N$ the set $\{1, 2, 3, \ldots \}$ of positive integers, and by $\P$ the set of all prime numbers. We use the letter $p$ in sums and products to indicate that they are over prime numbers. We let $(p_n)_{n \in \N}$ be the sequence of all prime numbers in increasing order. We let $\pi(N)$ be the number of prime numbers up to $N$, and, for $M \geq N$, we write $\pi(N,M) = \pi(M) - \pi(N)$. We write $\pi(N;a,q)$ for the number of prime numbers congruent to $a \bmod q$ up to $N$. Similarly, $\pi(N, M;a,q)$ denotes the number of primes between $N$ and $M$ that are congruent to $a \pmod q$. We write the greatest common divisor of some integers $a$ and $b$ as $(a,b)$. We write $\varphi(n)$ for Euler's totient function evaluated at $n$, and $\tau_j(n)$ for the number of ways to decompose $n$ as a product of $j$ positive integers. We also write $\Lambda(n)$ for the von Mangoldt function evaluated at $n$, that is
\begin{align*}
    \Lambda(n) = \begin{cases}
        \log p & \text{ if } n = p^k;\\
        0 & \text{ otherwise.}
    \end{cases}
\end{align*}

\noindent The indicator function of a set $S$ evaluated at $n$ is denoted by $\mathbf{1}_S(n)$. We shorten $e^{2\pi i x}$ as $e(x)$. We denote the $\ell$-dimensional Lebesgue measure of a measurable set $E \subset \R^{\ell}$ by $\lambda_{\ell}(E)$. We write $\langle \cdot , \cdot \rangle$ for the dot product of $\R^\ell$ and denote $\bar{0}\in\R^\ell$ the zero vector. For real-valued functions $F$ and $G$, we write $F(x) \ll G(x)$ or $F(X) = O(G(x))$ for $x \in S$ if there exists some constant $C>0$ such that $|F(x)| \leq C G(x)$ for all $x \in S$. We refer to the constant $C$ as the implied constant. If the constant $C$ depends on some parameters $\alpha, \beta, \ldots$, we write $F \ll_{\alpha, \beta, \ldots} G$ or $F = O_{\alpha, \beta, \ldots}(G)$. For $x \in \R$, we write $\lfloor x \rfloor$ and $\lceil x \rceil$ to denote the largest integer no larger than $x$ and the smallest integer no smaller than $x$ respectively. We write $\{ x \}$ for $x - \lfloor x \rfloor$. We also write $\lVert x \rVert$ for the distance between $x$ and the closest integer. We denote the cardinality of a set $I$ by $\# I$.

\section{Main ideas of the proof}\label{main_ideas}
The main ingredient of our proof of Theorem \ref{main_result} is a result of Maynard in \cite{Maynard_clusters} about bounded gaps between (consecutive) primes in sets satisfying some technical hypotheses. We first state a useful definition.
\begin{Definition}
    A family of linear polynomials $\{ L_1, \ldots, L_d \}$ with $L_i(n) = a_i n + b_i$, $a_i \in \N$, $b_i \in \N \cup \{ 0 \}$ is \define{admissible} if for every prime $p$, there exists some $n$ such that $\prod_{i=1}^d L_i(n) \not\equiv 0 \bmod{p}$. We say that set of integers $\{ h_1, \ldots, h_d \}$ is admissible if the family of linear polynomials $\{ n+ h_1, \ldots, n + h_d \}$ is admissible.
\end{Definition}

\begin{Theorem}[\cite{Maynard_clusters}, Theorem 3.1]\label{maynard}
    Let $\alpha > 0$ and $0 < \theta < 1$. Let $\mathcal{B}$ be a set of integers, $\mathcal{P}$ a subset of $\mathbb{P}$, and $\mathcal{L} = \{ L_1, \ldots, L_d\}$ an admissible family of $d$ linear polynomials, where $L_i(n) = a_i n + b_i$, $a_i \in \N$, $b_i \in \N \cup \{0\}$. Let $B$ and $N$ be positive integers. Assume that $B \leq N^\alpha$, $a_i, b_i \leq N^\alpha$ for $1 \leq i \leq d$, and $d \leq (\log N)^\alpha$. Then, there exists a constant $C$ depending only on $\alpha$ and $\theta$ such that the following holds when $d \geq C$:
    
    If there exists some $\delta > 1/\log d$ such that,
    \begin{enumerate}
        \item \label{H1} The set $\{a_i n + b_i : n \in \mathcal{B}, 1 \leq i \leq d\}$ contains a large enough proportion of primes:
        \[
         \frac{1}{d} \frac{\varphi(B)}{B} \sum_{L_i \in \mathcal{L}} \frac{\varphi(a_i)}{a_i} \sum_{\substack{N \leq n < 2N \\ a_i n + b_i \in \mathcal{P} }} \textbf{1}_{\mathcal{B}}(n) \geq \frac{\delta}{\log N} \sum_{N \leq n < 2N} \mathbf{1}_{\mathcal{B}}(n).
        \]
        \item \label{H2} $\mathcal{B}$ is well-distributed in arithmetic progressions:
        \[
        \sum_{\substack{q \leq N^{\theta} \\ (q,B) = 1}} \max_{0 \leq c \leq q-1} \left| \sum_{\substack{N \leq n < 2N \\ n\equiv c \bmod q}} \mathbf{1}_{\mathcal{B}}(n) - \frac{1}{q} \sum_{N \leq n < 2N} \mathbf{1}_{\mathcal{B}}(n) \right| \ll \frac{\sum_{N \leq n < 2N} \mathbf{1}_{\mathcal{B}}(n)}{(\log N)^{101d^2}}.
        \]
        \item  \label{H3} Primes represented by the linear forms in $\mathcal{L}$ are well-distributed in arithmetic progressions, that is, for any $L \in \mathcal{L}$, the following holds:
        \[
        \sum_{q \leq N^{\theta}} \max_{\substack{c: (L(c),q) = 1}} \left| \sum_{\substack{N \leq n < 2N \\ n\equiv c \bmod q \\ L(n) \in \mathcal{P}}} \mathbf{1}_{\mathcal{B}}(n) - \frac{1}{\varphi_L(q)} \sum_{\substack{N \leq n < 2N\\ L(n) \in \mathcal{P}} }\mathbf{1}_{\mathcal{B}}(n) \right| \ll \frac{\sum_{N \leq n < 2N, L(n) \in \mathcal{P}} \mathbf{1}_{\mathcal{B}}(n)}{(\log N)^{101d^2}}.
        \]
        where $\varphi_L(q) = \varphi(aq)/\varphi(a)$ for  $L(n) = an+b$.
        \item \label{H4} The set $\mathcal{B}$ is not too concentrated in any single arithmetic progression: for any $q \leq N^{\theta}$ and $0 \leq c \leq q-1$, the following holds:
        \[
        \sum_{\substack{N \leq n < 2N \\ n\equiv c \bmod q}} \mathbf{1}_{\mathcal{B}}(n) \ll \frac{1}{q} \sum_{N \leq n < 2N}\mathbf{1}_{\mathcal{B}}(n).
        \]
    \end{enumerate}
    Then,
    \[
    \#\{ n \in \mathcal{B} \cap [N, 2N) : \# \{ L_1(n), \ldots, L_d(n) \} \cap \mathcal{P} \geq C^{-1} \delta \log d \} \gg \frac{\sum_{N \leq n < 2N} 1_{\mathcal{B}}(n)}{(\log N)^d \exp(Cd)}.
    \]
    If, in addition, $\mathcal{P} = \P$, $k \leq (\log N)^{1/5}$, and $a_i = a \ll 1, b_i \leq (\log x) d^{-2}$, we have
    \begin{align*}
        &\#\{ n \in \mathcal{B} \cap [N, 2N) : \text{ at least } \lceil C^{-1} \delta \log d \rceil \text{ of }  \{ L_1(n), \ldots, L_d(n) \} \text{ are consecutive primes} \}\\
        &\qquad \gg \frac{\sum_{N \leq n < 2N} 1_{\mathcal{B}}(n)}{(\log N)^d \exp(Cd^5)}.\\
    \end{align*}
    In both estimates, the  implied constants do not depend on $N$ and $d$.
\end{Theorem}

When $B = 1$, $a_i = 1$ for all $i$, $0 < \theta < 1/2$, $\mathcal{B} = \N$ and $\mathcal{P} = \P$, Hypotheses 1-4 of Theorem \ref{maynard} are well known. Indeed, Hypothesis \ref{H1} is a consequence of the prime number theorem, Hypotheses \ref{H2} and \ref{H4} are obvious, and Hypothesis \ref{H3} follows from a form of the Bombieri-Vinogradov theorem.
\begin{Theorem}[Bombieri-Vinogradov]\label{Bombieri}
For all real numbers $0<\theta<\frac{1}{2}$ and $D>0$, and every integer $N \in \N$ large enough in terms of $\theta$ and $D$, we have
\begin{equation}\label{BVE0}
    \sum_{q\leq N^\theta}\max_{(a,q)=1}\left|\pi(N;a,q)-\frac{1}{\varphi(q)}\pi(N)\right|\ll_{D}\frac{N}{(\log N)^D}.
\end{equation}
\end{Theorem}
\begin{proof}
    This follows from \cite[Theorem 17.1]{iwaniec2021analytic} by partial summation.
\end{proof}
To prove Theorem \ref{main_result}, we show that for every $m \in \N$, we can find some $d > m$ and some integers $h_1 < h_2 < \ldots < h_d$ for which there exists a subset $\mathcal{B}\subset\{ n \in \N : \{F(n+h_1)\}, \ldots, \{F(n+h_d)\} \in U \}$ which satisfies the hypotheses of \cref{maynard} for the linear polynomials $L_i(n) = n + h_i$, $1 \leq i \leq d$. In particular, for every $N \in \N$ large enough, there exists some
\begin{equation*}\label{equ_set}
n \in [N,2N) \cap \mathcal{B}
\end{equation*}
for which $\{ n + h_1, \ldots, n + h_d \}$ contains at least $m$ consecutive primes. Notice that $n\in\mathcal{B}$ implies $n\in\{ n \in \N: \{F(n+h_1)\}, \ldots, \{F(n+h_d)\} \in U \}$ which in turn implies that these primes will necessarily lie in the set $\mathcal{A}$ from the statement of \cref{main_result}. 
We show that we can choose the integers $h_1, \ldots, h_d$ so that the aforementioned set $\mathcal{B}$ can be taken to be 
\begin{equation*}\label{equ_form_set}
    \mathcal{B} = \{ n \in \N : (\{F(n+h_1)\}, \ldots, \{F(n+h_\ell)\}) \in U_\ell \},
\end{equation*}
where $U_\ell$ is a semi-open box in $[0,1)^{\ell}$, and $\ell$ is a well chosen integer less than $d$. To verify the hypotheses of \cref{maynard}, we use the fact that for this choice of $\ell$, the sequence $(\{F(n+h_1)\}, \ldots, \{F(n+h_\ell)\})$ is well distributed in $[0,1)^\ell$. A quantitative measure of how well a finite sequence is distributed in $[0,1)^\ell$ is given by its discrepancy. 

\begin{Definition}
The discrepancy of a finite sequence $(\mathbf{x}_n)_{n=1}^N$ in $[0,1)^\ell$ is the quantity
    \[
    D\left((\mathbf{x}_n)_{n=1}^N\right) = \sup_{B \subset [0,1)^\ell} \left| \frac{1}{N} \sum_{n=1}^N \mathbf{1}_{B}(\mathbf{x}_n) - \lambda_\ell(B)\right|,
    \]
where the supremum is taken over semi-open boxes of the form $B = \prod_{i=1}^\ell [a_i,b_i)$, where $0\leq a_i < b_i < 1$ for each $i$.
\end{Definition}
The \Erdos-\Turan-Koksma inequality gives a bound on the discrepancy of a sequence in terms of exponential sums.
\begin{Theorem}[\Erdos-\Turan-Koksma inequality, \cite{drmota2006sequences}, Theorem 1.21] \label{discbd}
    For every $H\in \N$ and for every finite sequence $(\mathbf{x}_n)_{n = 1}^N$ in $[0,1)^\ell$, we have
    \[
    D\left((\mathbf{x}_n)_{n=1}^N\right) \ll_\ell  \left(\frac{1}{H+1}+\sum_{\substack{\mathbf{m} \in Z^{\ell} \setminus \{\bar{0}\} \\ \lVert \mathbf{m} \rVert_\infty \leq H}}\frac{1}{r(\mathbf{m})} \left|\frac{1}{N}\sum_{n=1}^N e(\langle \mathbf{x}_n, \mathbf{m} \rangle)\right| \right),
    \]
    where if $\ \mathbf{m} = (m_1, \ldots, m_\ell)$, we define $\lVert \mathbf{m} \rVert_\infty := \max_{1 \leq i \leq \ell} |m_i|$, and \\ $r(\mathbf{m}) := \prod_{i=1}^\ell \max\{|m_i|, 1\}$.
\end{Theorem}

In Section \ref{discr_integers}, we study the distribution of sequences of the form \\ $(\{ F(n+h_1) \}, \ldots, \{ F(n+h_\ell) \})$ in $[0,1)^\ell$, where $F$ is as in the statement of \cref{main_result}, $\ell = \lfloor \rho_r \rfloor + 1$, and where $n$ runs over the integers in some dyadic interval $[N,2N)$ or over the integers in $[N,2N)$ that lie in an arithmetic progression. This allows us to verify Hypotheses \ref{H2} and \ref{H4} of \cref{maynard} in the proof of \cref{main_result}. Then, in Section \ref{discr_primes}, we develop the tools to verify the Hypotheses \ref{H1} and \ref{H3} by studying the distribution of the same sequence but over primes in $[N,2N)$, as well as the primes $[N,2N)$ restricted to an arithmetic progression. The restriction $\rho_r \not \in \Z$ to the polynomial with real exponents $F(x) = \sum_{i=1}^r d_i x^{\rho_i}$ ensures the existence of good upper bounds on exponential sums of the form
\begin{align*}
    \sum_{N \leq n < 2N} e(\langle (F(n+h_1), \ldots, F(n+h_\ell)), \mathbf{m} \rangle)
\end{align*}
for $\mathbf{m}$ a vector of $\ell$ integers, which is crucial for our argument. This is explained in more detail at the beginning of Section \ref{lastsection}, and especially in \cref{rem_notinz}. Finally, in Section \ref{lastsection}, we combine the results of the preceding sections to prove Theorem \ref{main_result}.\\
\\

\section{Uniform distribution of values of polynomials with real exponents in arithmetic progressions}\label{discr_integers}

In this section, we prove results that will help us verify Hypotheses \ref{H2} and \ref{H4} of Theorem \ref{maynard} (\cref{l13} and \cref{hypothesis4}). To do this, we study the distribution of sequences of the form $\overline{F}(n) = (\{ F(n+h_1) \}, \ldots, \{ F(n+h_\ell) \})$ in $[0,1)^\ell$ for $F$ as in the statement of \cref{main_result}. \cref{discbd} allows us to bound the discrepancy of such sequences by exponential sums over sequences of the form $\langle \mathbf{m}, \overline{F}(n) \rangle$ where $\mathbf{m}$ is a nonzero vector of integers. We first prove in \cref{prop_lin_comb} that $\langle \mathbf{m},\overline{F} \rangle$ can be approximated by polynomial-like functions with real exponents. Then, in \cref{lem_BKM+_reformulated}, we reprove a bound on exponential sums from \cite{Bergelson2014}, and we specialize it to our setting in \cref{be} and \cref{corbdd}.\\

\begin{Proposition}\label{prop_lin_comb}
    Let $\mathcal{H} = \{ h_1, \ldots, h_\ell \}$ be a set of real numbers, with $h_1 < h_2 < \ldots < h_\ell$, and let $F(x) = \sum_{i=1}^r d_i x^{\rho_i}$ with $0 \leq \rho_1 < \ldots < \rho_r, \rho_r \not \in \Z, d_r \neq 0$ be a polynomial with real exponents, such that $\ell = \lfloor \rho_r \rfloor + 1$. Define $\overline{F}(x) := (F(x+h_1), \ldots, F(x+h_\ell))$. Then, for any $\eta > 0$, any integer $N$ and any $\mathbf{m} = (m_1, \ldots, m_\ell) \in \Z^\ell \setminus \{\overline{0}\}$ with $\max\limits_{1 \leq j \leq \ell} |m_j| \leq N^{\eta}$, we can write
    \begin{align*}
        \langle \mathbf{m} ,\overline{F}(x) \rangle = F_1(x) + O_{F, \mathcal{H}}(x^{\{\rho_r\} + \eta - 1}),
    \end{align*}
    where $F_1$ is a function of the form $F_1(x) = \sum_{n = n_0}^{\ell - 1}\sum_{i = 1}^r b_{n,i}x^{\rho_i - n}$, with the coefficient of the largest power of $x$ in the expression being $b_{n_0, r} = \left( \sum_{j=1}^\ell m_j h_j^{n_0} \right) d_r \binom{\rho_r}{n_0}$ for some integer $0 \leq n_0 \leq \ell - 1$. Moreover, the coefficients $b_{n,j}$ are all uniformly bounded by
    \begin{align*}
        \ell^2 r \max_{\substack{0 \leq n \leq \ell - 1 \\ 1 \leq j \leq \ell \\ 1 \leq i \leq r}} \left(\left|h_j^n d_i \binom{\rho_i}{n}\right|\right) N^\eta,
    \end{align*}
    where 
    \begin{align*}
        \binom{\rho}{0} = 1, \qquad  \binom{\rho}{n} = \frac{\rho(\rho - 1) \ldots (\rho - n + 1)}{n!} \quad \text{ for } n \geq 1.
    \end{align*}
\end{Proposition}
\begin{proof}
 By Taylor's theorem, we have that that for $|y| < 1$ and $\rho \in \R$,
    \begin{align*}
        (1+y)^\rho = \sum_{n=0}^{\ell - 1} \binom{\rho}{n} y^n  + O_{\rho}(y^\ell).
    \end{align*}
    It follows that for $j = 1, 2, \ldots, \ell$ and $x > h_\ell$, $x\geq 1$,
    \begin{align*}
        F(x + h_j) &= \sum_{i = 1}^r d_i (x + h_j)^{\rho_i}\\
        &= \sum_{i = 1}^r d_i x^{\rho_i} \left( \sum_{n=0}^{\ell - 1} \binom{\rho_i}{n} \left( \frac{h_j}{x} \right)^n  + O_{\rho_i}\left(\left( \frac{h_j}{x} \right)^\ell\right) \right)\\
        &= \sum_{i = 1}^r d_i   \sum_{n=0}^{\ell - 1} \binom{\rho_i}{n} h_j^n x^{\rho_i - n}  + O_{F, \mathcal{H}} \left(x^{\{ \rho_r \} - 1} \right).  \\
    \end{align*}
    For a fixed $\mathbf{m}=(m_1,...,m_\ell)\in \Z^\ell\backslash \{\bar{0}\}$, we can write
    \begin{align*}
        \langle \mathbf{m} ,\overline{F}(x) \rangle &= \sum_{j=1}^{\ell}m_jF(x+h_j)\\
        &=\sum_{j=1}^{\ell}m_j\sum_{i = 1}^r d_i   \sum_{n=0}^{\ell - 1} \binom{\rho_i}{n} h_j^n x^{\rho_i - n} + O_{F,\mathcal{H}} \left(x^{\{ \rho_r \} - 1} \sum_{j=1}^\ell|m_j| \right)\\
        &=    \sum_{n=0}^{\ell - 1}  \left( \sum_{j=1}^\ell m_j h_j^n \right) \sum_{i = 1}^r d_i\binom{\rho_i}{n} x^{\rho_i - n} +O_{F, \mathcal{H}}(x^{\{\rho_r\} + \eta - 1}).
    \end{align*}
    Define
    \begin{align*}
        F_1(x) &= \sum_{n=0}^{\ell - 1}  \left( \sum_{j=1}^\ell m_j h_j^n \right) \sum_{i = 1}^r d_i\binom{\rho_i}{n} x^{\rho_i - n} \\
        &= \sum_{n = 1}^{\ell - 1}\sum_{i = 1}^r b_{n,i}x^{\rho_i - n}.
    \end{align*}
    The term with the highest power of $x$ in this expression is 
    \begin{align*}
        b_{n_0,r} x^{\rho_r - n_0} = \left( \sum_{j=1}^\ell m_j h_j^{n_0} \right) d_r \binom{\rho_r}{n_0} x^{\rho_r - n_0},
    \end{align*}
    where $n_0$ is the smallest integer $0 \leq n_0 \leq \ell - 1$ such that $\sum_{j=1}^\ell m_j h_j^{n_0}$ is nonzero. This $n_0$ exists since $(m_1, \ldots, m_\ell)$ is a non-zero vector of integers and if we write
    $$\nu(h) = \begin{pmatrix}
        1\\
        h\\
        \vdots\\
        h^{\ell - 1}
    \end{pmatrix} \in \R^{\ell},$$
    then the vectors $\nu(h_1), \ldots, \nu(h_\ell)$ are all linearly independent. This follows, for example, from the fact that they are the columns of the Vandermonde matrix
    \begin{align*}
        V= \begin{pmatrix}
            1 & 1 & 1 & \cdots & 1\\
            h_1 & h_2 & h_3 & \cdots & h_\ell\\
            h_1^2 & h_2^2 & h_3^2 & \cdots & h_\ell^2\\
            \vdots & \vdots & \vdots & \ddots & \vdots\\
            h_1^{\ell - 1} & h_2^{\ell - 1} & h_3^{\ell - 1} & \cdots & h_\ell^{\ell - 1}\\
        \end{pmatrix},
    \end{align*}
    and $\det(V) = \prod\limits_{1 \leq i < j \leq \ell} (h_j - h_i) \neq 0$.\\
    
    Also, the absolute value of any of the coefficients of $F_1(x)$ is certainly bounded by the sum of the absolute values of all the coefficients, which in turn, by the triangle inequality, is bounded by
    \begin{align*}
        \sum_{n=0}^{\ell - 1}  \left( \sum_{j=1}^\ell |m_j| |h_j|^n \right) \sum_{i = 1}^r |d_i| \left|\binom{\rho_i}{n}\right|
        &\leq \ell^2 r \max_{\substack{0 \leq n \leq \ell - 1 \\ 1 \leq j \leq \ell \\ 1 \leq i \leq r}} \left(\left|h_j^n d_i \binom{\rho_i}{n}\right|\right) N^\eta.
    \end{align*}
    
\end{proof}

Next, we prove an upper bound on exponential sums of the form
\begin{align*}
    \sum_{X \leq x < 2X} e(f(x) + P(x)),
\end{align*}
where $P$ is a polynomial whose degree is bounded, and $f$ is a function whose derivatives satisfy certain growth conditions.

\begin{Lemma}\label{lem_BKM+_reformulated}
Let $k, q \in \N$, and set $K = 2^k$ and $Q = 2^q$. Let $0< c_1<c_2$. For every $X \in \N$, $X \geq K + 1$, every real number $g \in [1, X^{q+2}]$, every real-valued function $f(x)$ that is $(q + k + 2)$-times continuously differentiable on $[X, 2X]$ and satisfies $c_1gX^{-r}\leq \left| f^{(r)}(x)\right| \leq c_2gX^{-r}$ for $r = 1,\ldots q + k + 2$, and every real polynomial $P(x)$ of degree at most $k$, one has
    \[
    \left|\sum_{X \leq x < 2X} e(f(x) + P(x)) \right| \ll_{c_1,c_2,q,k}  X^{1- \frac{1}{K}} + X \left( \frac{(\log X)^k}{g} \right)^{1/K} + X \left(\frac{g}{X^{q+2}}\right)^{\frac{1}{4QK-2K}}.
    \]
\end{Lemma}
\begin{proof}
    This lemma is a reformulation of \cite[Lemma 2.5]{Bergelson2014}. We reprove it here for the convenience of the reader.
    Lemma~2.1 of \cite{Bergelson2014} with $H_i = \frac{X}{K}$ for $i = 1, \ldots, k$ and $S = \left|\sum_{X \leq x < 2X} e(f(x) + P(x)) \right|$ gives us 
    \begin{align*}
    \left( \frac{S}{X-1}\right)^K &\leq 8^{K-1} \left( \frac{1}{H_1^{K/2}} + \ldots +  \frac{1}{H_k} +  \frac{1}{H_1\ldots H_k(X-1)} \sum_{h_1=1}^{H_1} \ldots \sum_{h_k=1}^{H_k}\left| \sum_{x \in I(\mathbf{h})} e(f_1(x))\right| \right) \\
    &\ll_k \frac{1}{X} + \frac{1}{H_1\ldots H_k(X-1)} \sum_{h_1=1}^{H_1}\sum_{h_2=1}^{H_2} \ldots \sum_{h_k=1}^{H_k}\left| \sum_{x \in I(\mathbf{h})} e(f_1(x))\right| \label{ineq_S}.
    \end{align*}
    where $f_1(x) = h_1 \ldots h_k \left[ \int_0^1 \ldots \int_0^1 \frac{\partial^k}{\partial x^k}f(x + \langle\mathbf{h},\mathbf{t}\rangle) \d\mathbf{t} + a_kk!\right], \mathbf{h} = (h_1,\ldots h_k)$, $a_k$ is the coefficient of $x^k$ in $P(x)$ (if the degree of $P(x)$ is less than $k$, then $a_k = 0$), and $I(\mathbf{h}) = (X, 2X - 1 - h_1 - \ldots - h_k]$. Here, the variable $\mathbf{t}$ is a vector in $\left[0,1\right]^k$, and the measure is the usual Lebesgue measure.
    
     We have that $f_1(x)$ has $q+2$ continuous derivatives on $I(\mathbf{h})$, and
    \[
    f_1^{(r)}(x) =  h_1 \ldots h_k \left[ \int_0^1 \ldots \int_0^1 \frac{\partial^{k+r}}{\partial x^{k+r}}f(x + \langle\mathbf{h},\mathbf{t}\rangle) \d\mathbf{t} \right]
    \]
    for $r = 1 , \ldots, q + 2$. This, combined with the hypothesis that $c_1gX^{-r}\leq \left| f^{(r)}(x)\right| \leq c_2gX^{-r}$, gives us the inequalities
    \[
    c_1gh_1 \ldots h_k X^{-(k+r)} \leq \left|f_1^{(r)}(x)\right| \leq c_2gh_1 \ldots h_k X^{-(k+r)}
    \]
 for $r = 1, \ldots, q + 2$. Hence, the function $f_1(x)$ satisfies the conditions of Lemma 2.2 in \cite{Bergelson2014}, with $G = \frac{gh_1 \ldots h_k}{X^k}$ and $I = I(\mathbf{h})$. It follows that
    \[
    \left| \sum_{x \in I(\mathbf{h})} e(f_1(x))\right| \ll_{c_1,c_2,q}  G^{\frac{1}{4Q-2}}X^{1 - \frac{q+2}{4Q-2}} + \frac{X}{G}.
    \]
   Hence, we obtain
    \begin{align*}
        \left( \frac{S}{X-1} \right)^K &\ll_{c_1,c_2,q,k}  \frac{1}{X} + \frac{1}{H_1\ldots H_k(X-1)}\sum_{h_1=1}^{H_1}\sum_{h_2=1}^{H_2} \ldots \sum_{h_k=1}^{H_k} \left( G^{\frac{1}{4Q-2}}X^{1 - \frac{q+2}{4Q-2}} + \frac{X}{G}\right).
    \end{align*}
     Now, we have $G = \frac{gh_1 \ldots h_k}{X^k} \leq \frac{gH_1\ldots H_k}{X^k} = \frac{g}{(K)^k}$. With this upper bound, the first term in the innermost sum is independent of the $h_i$'s. To evaluate the second term, we have 
    \begin{align*}
    \sum_{h_1=1}^{H_1} \ldots \sum_{h_k=1}^{H_k} \frac{1}{G}
    = \frac{X^k}{g}\sum_{h_1=1}^{H_1} \ldots \sum_{h_k=1}^{H_k} \frac{1}{h_1\ldots h_k}
    \ll_k \frac{X^k}{g} \log H_1 \ldots \log H_k \leq \frac{X^k}{g} (\log X)^k.
    \end{align*}
     By using all of these bounds and substituting $H_i = \frac{X}{K},$ we get 
    \begin{align*}
    \left( \frac{S}{X-1} \right)^K &\ll_{c_1,c_2,q,k}  \frac{1}{X} + \left(\frac{g}{K^kX^{q+2}}\right)^{\frac{1}{4Q-2}} +\frac{K^k}{g} (\log X)^k\\
    &\ll_{c_1,c_2,q,k} \frac{1}{X} + \left(\frac{g}{X^{q+2}}\right)^{\frac{1}{4Q-2}} +\frac{(\log X)^k}{g}.
    \end{align*}
    Thus,
    \begin{align*}
    S \ll_{c_1,c_2,q,k}  X^{1- \frac{1}{K}} + X \left(\frac{g}{X^{q+2}}\right)^{\frac{1}{4QK-2K}} + X \left( \frac{(\log X)^k}{g} \right)^{1/K}.
    \end{align*}
\end{proof}
For the rest of this paper, we fix $\eta = \min \bigl\{\frac{1 - \{\rho_r\}}{2}, \rho_r - \rho_{r-1} \bigr\}/2$. We now use \cref{prop_lin_comb} together with \cref{lem_BKM+_reformulated} to get an upper bound on the exponential sums that appear when we apply the \Erdos-\Turan-Koksma inequality (\cref{discbd}) later in the section.
    \begin{Proposition} \label{be}
    Let $k,r\in\N$ and fix a polynomial with non-integer exponents $f(x)=\sum_{j=1}^r d_jx^{\rho_j}$ with $d_r\neq 0$ and $0 < \rho_1 < \rho_2 < \ldots < \rho_r$. Let $\ell = \lfloor \rho_r \rfloor + 1$, and let $\mathcal{H} = \{ h_1, \ldots, h_\ell \}$ be a set of integers with $h_1 < h_2 < \ldots < h_\ell$. Define $\bar f(x) := (f(x+h_1), \ldots, f(x+h_\ell))$. Then, there exists some $\alpha>0$ depending only on $k$ and $\rho_r$ such that uniformly over all real polynomials $P$ of degree at most $k$, all $N\in\N$ large enough in terms of $f, \mathcal{H}$ and $k$, and all $\mathbf{m} = (m_1, \ldots, m_\ell) \in \Z^\ell \setminus\{\bar{0}\}$ with $\max\limits_{1 \leq j \leq \ell} |m_j| \leq N^{\eta}$, 
    \begin{equation}
    \label{eqn_uni_dis_of_gen_poly_in_arith_prog_1}
    \left|\sum_{N \leq n < 2N}e(\langle \mathbf{m}, \bar f(n)\rangle + P(n)) \right| \ll_{f,\mathcal{H}, k} N^{1-\alpha}.
    \end{equation}
\end{Proposition}
\begin{proof}
By Proposition \ref{prop_lin_comb}, we can write
\begin{align*}
    \langle \mathbf{m}, \bar f(x)\rangle = f_1(x) + R(x),
\end{align*}
where $f_1$ is of the form $\sum_{n = n_0}^{\ell - 1}\sum_{i = 1}^r b_{n,i}x^{\rho_i - n}$, with $b := b_{n_0,r} =  \left( \sum_{j=1}^\ell m_j h_j^{n_0} \right) d_r \binom{\rho_r}{n_0} \neq 0$,  and 
$$ R(x) \ll_{f,\mathcal{H}} x^{\{\rho_r\} - 1 + \eta}. $$
Notice that our choice of $\eta$ ensures that $\{\rho_r\} - 1 + \eta < 0$. By the triangle inequality, we have that for any polynomial $P$, 
\begin{align*}
    \left|\sum_{N \leq n < 2N}e(\langle \mathbf{m}, \bar f(n)\rangle + P(n)) \right| &\leq  \left|\sum_{N \leq n < 2N}e(f_1(n) + P(n)) \right| \\
    &\qquad+ \left|\sum_{N \leq n < 2N}e(f_1(n) + P(n) + R(n)) - \sum_{N \leq n < 2N}e(f_1(n) + P(n)) \right|\\
    &= \left|\sum_{N \leq n < 2N}e(f_1(n) + P(n)) \right|\\
    &\qquad+ \left|\sum_{N \leq n < 2N} e(f_1(n) + P(n))\left(e(R(n)) -1 \right) \right| .
\end{align*}
Call the first term $S_1$ and the second term $S_2$. We have that for any $N \geq 1$, 
\begin{align*}
    S_2 &\leq \sum_{N \leq n < 2N}\left| e(f_1(n) + P(n)) \right| \left| e(R(n)) -1 \right|\\
    &= \sum_{N \leq n < 2N} \left| e(R(n)) -1 \right|\\
    &= \sum_{N \leq n < 2N}\left| \sum_{k=1}^\infty \frac{(2\pi i)^k}{k!}(R(n))^k  \right|\\
    &\ll_{f,\mathcal{H}} \sum_{k=1}^\infty \frac{(2\pi)^k}{k!} \sum_{N \leq n < 2N} n^{\{\rho_r\} - 1 + \eta}\\
    &\ll_{f,\mathcal{H}} N^{\{\rho_r\} + \eta}\\
    &\ll_{f,\mathcal{H}} N^{\frac{1+ \{\rho_r\}}{2}}.\\
\end{align*}

To get an upper bound on $S_1$, we use Lemma \ref{lem_BKM+_reformulated}. The function $f_1$ can be written as
\begin{align*}
    f_1(x) = b x^{\rho_r - n_0} + \sum_{i = 1}^{r'} b_{i} x^{\rho_i'},
\end{align*}
where $\rho_1' < \rho_2' < \ldots < \rho_{r'}' < \rho_r - n_0$, $|b|,|b_i| \ll_{f, \mathcal{H}} N^\eta$, and $r' \leq \ell r$. This function is smooth on $(0,\infty)$, and for every $s \in \N \cup \{0 \}$, we have that
\begin{align*}
    f_1^{(s)}(x) = b \mathrm{P}^{\rho_r-n_0}_{s}x^{\rho_r - n_0 - s} + \sum_{i = 1}^{r'} b_{i} \mathrm{P}^{\rho^\prime_i}_{s} x^{\rho_i' - s},
\end{align*}
where $\mathrm{P}_{n}^{\rho} = \rho(\rho - 1)\ldots (\rho-n+1)$ for $\rho \in \R$ and $n \in \Z^+$. Using this fact together with the uniform bound for the $b_i$'s from the conclusion of \cref{prop_lin_comb}, we get that for each $0 \leq s \leq \ell + k + 1$,
\begin{align*}
    \left| \sum_{i = 1}^{r'} b_{i} \mathrm{P}^{\rho^\prime_i}_{s}x^{\rho_i' - s} \right| 
    &\ll_{f,\mathcal{H},k} x^{\max\{\rho_{r - 1} - n_0, \rho_r - n_0 - 1\} + \eta - s}.
\end{align*}
Since $\eta < \min \bigl\{\frac{1 - \{\rho_r\}}{2}, \rho_r - \rho_{r-1} \bigr\},$ there exists some integer $N_0 = N_0(f,\mathcal{H}, \eta)$ such that for any $N_0\leq N \leq x \leq 2N$ and any $0 \leq s \leq \ell + k + 1$, we have
\begin{align*}
    c_1 |b| N^{\rho_r - n_0} N^{-s} \leq |f_1^{(s)}(x)| \leq c_2  |b| N^{\rho_r - n_0} N^{-s},
\end{align*}
where 
\begin{align*}
    c_1 = \frac{1}{2} \min_{\substack{0 \leq n \leq \ell - 1 \\ 0 \leq t \leq \ell + k + 1}} \left| \mathrm{P}^{\rho_r - n}_{t} \right|,\qquad
    c_2 = 2^{\rho_r + 1} \max_{\substack{0 \leq n \leq \ell - 1 \\ 0 \leq t \leq \ell + k + 1}} \left| \mathrm{P}^{\rho_r - n}_{t} \right|.
\end{align*}
Notice that since $\rho_r \not \in \Z$, we have that $c_1, c_2 > 0$. Also, notice that $c_1$ and $c_2$ depend only on $f$ and $k$. Therefore, by Lemma \ref{lem_BKM+_reformulated} with $k$ as the degree of the polynomial $P(x)$ and $q = \ell - 1$, there exists some constant $C' > 0$ that depends only on $f$ and $k$ such that for every polynomial $P$ of degree at most $k$ and every $N \geq N_0$,
\begin{equation}\label{ineq}
    S_1 \ll_{f, \mathcal{H},k} \left( N^{1-\frac{1}{2^k}} + N\left( \frac{(\log N)^k}{bN^{\{\rho_r\}}}\right)^{1/2^k} + N \left( \frac{|b|N^{\rho_r}}{N^{\ell+1}}\right)^{\frac{1}{2^{k+\ell + 1} - 2^k}}\right).
\end{equation}
Since $b = \left( \sum_{j=1}^\ell m_j h_j^{n_0} \right) d_r \binom{\rho_r}{n_0} \neq 0$ for some $n_0 \in \{0 , \ldots, \ell - 1\}$ where $\mathbf{m} = (m_1, \ldots, m_\ell) \in \Z^{\ell} \setminus\{\overline{0}\}$ and $\max\limits_{1 \leq j \leq \ell} |m_j| \leq N^{\eta}$, we have
\begin{align*}
    1 \ll_{f,\mathcal{H}} |b| = \left( \sum_{j=1}^\ell m_j h_j^n \right) d_r \binom{\rho_r}{n} \ll_{f,\mathcal{H}} N^\eta < N.
\end{align*}
Also, for any $\epsilon > 0$, $\left(\log N \right)^{k/K} \ll_{k,\varepsilon} N^{\epsilon}$. Combining all of this, we get 
\[
S_1 \ll_{f, \mathcal{H}, k, \varepsilon}  \left( N^{1-\frac{1}{2^k}} + N^{1 - \left(\frac{\{\rho_r\}}{2^k} - \varepsilon \right)}+ N^{1 - \frac{1- \{\rho_r\}}{2^{k + \ell + 1} - 2^{k+1}}} \right).
\]
Choosing $\epsilon > 0$ small enough and combining this with the estimate on $S_2$ yields the desired conclusion for some $\alpha>0$ depending only on $f$ and $k$.\\
\end{proof}
The uniformity over all polynomials of a bounded degree in \cref{be} allows us to restrict the sum on the left side of \eqref{eqn_uni_dis_of_gen_poly_in_arith_prog_1} to a certain congruence class without changing the upper bound.

\begin{Corollary} \label{corbdd}
Under the hypotheses of Proposition \ref{be}, uniformly over all real polynomials $P$ of degree at most $k$, all $N\in\N$ large enough depending on $f$, $\mathcal{H}$ and $k$, all $\mathbf{m} = (m_1, \ldots, m_\ell) \in \Z^\ell\setminus\{\bar{0}\}$ with $\max\limits_{1 \leq j \leq \ell} |m_j| \leq N^{\eta}$, all $q\in\N$, and all $a\in\{0,1,\ldots,q-1\}$, we have
    \begin{equation*}
    \left| \sum_{\substack{N\leq n < 2N\\ n\equiv a\bmod q}}e(\langle \mathbf{m}, \bar f(n)\rangle + P(n)) \right| \ll_{f, \mathcal{H}, k} N^{1-\alpha},
    \end{equation*}
    where $\alpha$ and the implied constant are the same as in Proposition \ref{be}.
\end{Corollary}
\begin{proof}

Recall that
\[
\mathbf{1}_{q\Z+a}(n)=\frac{1}{q}\sum_{0 \leq t \leq q-1} e\left( \frac{(n-a)t}{q}\right).
\]
Therefore, we have
\begin{align*}
\left| \sum_{\substack{N\leq n < 2N \\ n\equiv a\bmod q}} e(\langle \mathbf{m}, \bar f(n)\rangle+P(n)) \right|
&=
\left| \sum_{N\leq n < 2N } e(\langle \mathbf{m}, \bar f(n)\rangle+P(n)) \mathbf{1}_{q\Z+a}(n) \right|
\\
&=
\left| \sum_{N\leq n < 2N } e(\langle \mathbf{m}, \bar f(n)\rangle+P(n)) \left(\frac{1}{q}\sum_{0 \leq t \leq q-1} e\left( \frac{(n-a)t}{q}\right)\right) \right|
\\
&\leq \frac{1}{q}\sum_{0 \leq t \leq q-1}
\left| \sum_{N\leq n < 2N } e\left(\langle \mathbf{m}, \bar f(n)\rangle+P(n)+\frac{(n-a)t}{q}\right) \right|
\\
&= \frac{1}{q}\sum_{0 \leq t \leq q-1}\left| \sum_{N\leq n < 2N } e\left(\langle \mathbf{m}, \bar f(n)\rangle+P(n)+\frac{tn}{q}\right) \right|.
\end{align*}
Hence, the result follows from \cref{be}.\\
\end{proof}

Now, we prove some key results that will allow us to prove Hypotheses $\ref{H2}$ and \ref{H4} of \cref{maynard} in Section \ref{lastsection}. In fact, the following result is stronger than what is needed for Hypothesis \ref{H2}.

\begin{Proposition}\label{l13}
Let $P(x)$ be a real polynomial of degree at most $k \geq 1$ and let $f(x) = \sum_{j=1}^r d_j x^{\rho_j}$ be a polynomial with non-integer exponents with $0 < \rho_1 < \rho_2 < \ldots < \rho_r$ and $d_r \neq 0$. Define $\ell =\lfloor \rho_r\rfloor+1$. Then, there exist some constants $0 < \beta, \theta < 1$, depending only on $f(x)$ and $k$ such that the following holds:

Let $\mathcal{H} = \{h_1, \ldots h_{\ell}\}$ be a set of integers with $h_1 < \ldots < h_\ell$, and let $U_\ell = \prod_{i=1}^\ell [u_i, v_i)$ with $0 \leq u_i < v_i \leq 1$. Then, for any integer $N \in \N$ such that $h_\ell \leq N$, if we write
\[
\mathcal{B}=\{n\in\N:(\{f(n+h_1)+P(n + h_1)\}, \ldots, \{f(n+h_\ell)+P(n + h_\ell)\})\in U_\ell \},
\]
 then,
    \begin{equation}
    \label{eqn_disc_of_gen_poly_in_arith_prog_1}
        \sum_{q\leq N^{\theta}} \max_{0 \leq c \leq q-1} \left|\sum_{\substack{N\leq n < 2N \\ n \equiv c \bmod q}}\mathbf{1}_{\mathcal{B}}(n)-\frac{1}{q} \sum_{N\leq n < 2N}\mathbf{1}_{\mathcal{B}}(n)\right| \ll_{f,\mathcal{H}, k, U_\ell} \frac{1}{N^\beta}\sum_{N\leq n < 2N}\mathbf{1}_{\mathcal{B}}(n).
    \end{equation}
\end{Proposition}
\begin{proof}

Let $\mathbf{x}_n = (\{f(n+h_1) + P(n+h_1)\}, \ldots \{f(n + h_{\ell}) + P(n + h_{\ell})\})$. By $\ell$-dimensional \Erdos-Turán-Koksma inequality (Theorem \ref{discbd}), we have
\begin{align*}
D\left((\bar{f}(n)+\bar{P}(n))_{n \in [N,2N)}\right) \ll_\ell \left( \frac{1}{H + 1} + \sum_{0 < ||\mathbf{m}||_{\infty} \leqslant H} \frac{1}{r(\mathbf{m})}\left| \frac{1}{N} \sum_{N \leqslant n \leqslant 2N} e\left( \langle \mathbf{m},\mathbf{x}_n\rangle\right)\right|\right)
\end{align*}
 for any integer $H$, where $J$ is the set of all semi open rectangles of the form $\prod_{i=1}^{\ell}[a_i, b_i)$. We take $H = \lfloor N^{\eta} \rfloor$. Notice that we have the following:
\begin{align*}
    \left|\sum_{N \leqslant n \leqslant 2N} e\left(\langle \mathbf{m},\mathbf{x}_n\rangle \right)\right| 
    &= \left|\sum_{N \leqslant n \leqslant 2N} e\left(\langle \mathbf{m},\bar f(n) + \overline{P}(n) \rangle \right)\right| \\
    &= \left|\sum_{N \leqslant n \leqslant 2N} e\left(\langle(\mathbf{m},\bar f(n)\rangle+\langle \mathbf{m},\overline{P}(n)\rangle\right)\right|,
\end{align*}
where $\bar f(n) = (f(n + h_1), \ldots f(n + h_\ell))$ and $\overline{P}(n) = (P(n + h_1), \ldots P(n + h_\ell))$.
Since $\langle \mathbf{m},\overline{P}(n)\rangle$ is again a polynomial of degree at most $k$, by using \cref{be} and the fact that $0 < ||\mathbf{m}||_{\infty} \leqslant H \leqslant N^{\eta}$, we get that 
\begin{align*}
    D\left((\bar{f}(n)+\bar{P}(n))_{n \in [N,2N)}\right)
    &\ll_{f,\mathcal{H}, k} \left( \frac{1}{H + 1} + \sum_{0 < ||\mathbf{m}||_{\infty} \leqslant H} \frac{1}{r(\mathbf{m})}\frac{1}{N} N^{1 - \beta''}\right) \\
    &=\left( \frac{1}{H + 1} + N^{-\beta''}\sum_{0 < ||\mathbf{m}||_{\infty} \leqslant H}\frac{1}{r(\mathbf{m})}\right)
\end{align*}
for some $0 < \beta'' < 1/2$, depending only on $f(x)$ and the degree of $P(x)$.

Now, it remains to bound the sum $\sum\limits_{\substack{0 < ||\mathbf{m}||_{\infty} \leqslant H}}\frac{1}{r(\mathbf{m})}$. We have that
    \begin{align*}
    \sum_{0 < ||\mathbf{m}||_{\infty} \leqslant H}\frac{1}{r(\mathbf{m})} 
    &= \sum_{\substack{L \subset \{1,\ldots \ell\} \\ L \neq \varnothing}} \sum_{\substack{0 < |m_j| \leq H \\ j \in L}} \prod_{j \in L}\frac{1}{|m_j|}\\ 
    &=\sum_{\substack{L \subset \{1,\ldots \ell\} \\ L \neq \varnothing}} \prod_{j \in L} \sum_{0 < |m_j| \leqslant H} \frac{1}{|m_j|}\\
    &\ll \sum_{\substack{L \subset \{1,\ldots \ell\} \\ L \neq \varnothing}} \left( \log H \right)^{|L|},
    \end{align*}
    so, we have
    \[
    \sum_{0 < ||\mathbf{m}||_{\infty} \leqslant H}\frac{1}{r(\mathbf{m})} \ll_\ell \left( \log H \right)^{\ell} \ll_{f, k} N^{\beta''/2}.
\]
Putting everything together, we get \[
    D\left((\bar{f}(n)+\bar{P}(n))_{n \in [N,2N)}\right) \ll_{f,\mathcal{H},k} N^{-\beta'}
\]
for $\beta' = \min \{\eta, \beta''/2\}$. In particular, by the definition of $\mathcal{B}$ and since $U_\ell = \prod_{i=1}^\ell[u_i, v_i)$, it follows from the definition of the discrepancy that
\[
\left| \frac{1}{N}\sum_{N\leq n < 2N} \mathbf{1}_{\mathcal{B}}(n)-\lambda_\ell(U_\ell)\right| \ll_{f,\mathcal{H}, k} N^{-\beta'}, \]
i.e.
\begin{equation}\label{equ_no_congruence}
\sum_{N\leq n < 2N} \mathbf{1}_{\mathcal{B}}(n) = N\lambda_\ell(U_\ell) + O_{f,\mathcal{H}, k}(N^{1-\beta'}).
\end{equation}
Fix some $0 < \theta < 1$ to be determined later. Let $q,c$ be integers with $0 \leq c < q \leq N^{\theta}$, and define $\mathbf{y}_n = (\{f(qn+c+h_1) + P(qn+c+h_1)\}, \ldots \{f(qn+c+h_{\ell}) + P(qn+c+h_{\ell})\}) $. Note that we have $$\sum\limits_{\substack{N\leq n < 2N \\ n \equiv c \bmod q}} \mathbf{1}_{\mathcal{B}}(n) = \sum\limits_{\substack{N \leq nq+c < 2N}} \textbf{1}_{U_\ell}(\mathbf{y}_n) = \sum\limits_{n \in I} \textbf{1}_{\mathcal{B'}}(n),$$
where
$\mathcal{B'} = \{ n \in \N : \mathbf{y}_n \in U_\ell \}$ and
$$ I = \left\{\left \lceil \frac{N-c}{q} \right \rceil, \left \lceil \frac{N-c}{q} \right \rceil + 1, \ldots,\left \lceil \frac{2N-c}{q} \right \rceil  -1 \right\}.$$
Hence, by applying the Erdös-Turán-Koksma inequality, but this time to the sequence $(\mathbf{y}_n)_{n \in I}$, we get that for any $H\in \N$,
\begin{align*}
    D\left(\left(\bar{f}(n)+\bar{P}(n)\right)_{\substack{{n \in [N,2N)} \\ n \equiv c \bmod{q}}}\right) &\ll_\ell  \left( \frac{1}{H + 1} + \sum_{0 < ||\mathbf{m}||_{\infty} \leqslant H} \frac{1}{r(\mathbf{m})}\left| \frac{1}{\# I} \sum_{n \in I} e\left(\langle \mathbf{m},\mathbf{y}_n\rangle\right)\right|\right) \\
    &=  \left( \frac{1}{H + 1} + \sum_{0 < ||\mathbf{m}||_{\infty} \leqslant H} \frac{1}{r(\mathbf{m})}\left| \frac{1}{\# I} \sum_{\substack{N\leq n < 2N \\ n \equiv c \bmod q}}e\left(\langle \mathbf{m},\mathbf{x}_n\rangle\right)\right|\right).
\end{align*}
Again, by taking $H = \lfloor N^\eta \rfloor$, using \cref{corbdd} and following the exact same treatment as above, we get that
\begin{equation}\label{equ_congruence}
\sum\limits_{\substack{N\leq n  < 2N \\ n \equiv c \bmod q}} \mathbf{1}_{\mathcal{B}}(n) = \sum_{N\leq n < 2N} \mathbf{1}_{\mathcal{B'}}(n) = \#I \lambda_\ell(U_\ell) + O_{f, \mathcal{H}, k}(N^{1-\beta'})
\end{equation}
for $\beta'>0.$
Using that and the fact that $\#I - \frac{N}{q} = O(1),$ we get 
\[
\left|\sum_{\substack{N\leq n < 2N \\ n \equiv c \bmod q}}\textbf{1}_{\mathcal{B}}(n)-\frac{1}{q} \sum_{N\leq n < 2N}\textbf{1}_{\mathcal{B}}(n)\right| \ll_{f,\mathcal{H}, k, U_\ell} N^{1-\beta'}. 
\]
Hence, we get 
$$
\sum_{q\leq N^{\theta}} \max_{0 \leq c \leq q-1} \left|\sum_{\substack{N\leq n < 2N \\ n \equiv c \bmod q}}\textbf{1}_{\mathcal{B}}(n)-\frac{1}{q} \sum_{N\leq n < 2N}\textbf{1}_{\mathcal{B}}(n)\right| \ll_{f, \mathcal{H}, k, U_\ell} N^{1-\beta' + \theta}.
$$
Taking $\theta = \beta = \frac{\beta'}{2}$ yields
$$
\sum_{q\leq N^{\theta}} \max_{0 \leq c \leq q-1} \left|\sum_{\substack{N\leq n < 2N \\ n \equiv c \bmod q}}\textbf{1}_{\mathcal{B}}(n)-\frac{1}{q} \sum_{N\leq n < 2N}\textbf{1}_{\mathcal{B}}(n)\right| \ll_{f,\mathcal{H},k, U_\ell} N^{1 - \beta} \\
\ll_{f,\mathcal{H},k,U_\ell} \frac{1}{N^{\beta}} \sum_{N\leq n < 2N} \mathbf{1}_{\mathcal{B}}(n).
$$
\end{proof}

Finally, the following result will help us prove Hypothesis \ref{H4} of \cref{maynard} in Section \ref{lastsection}.
\begin{Corollary}{\label{hypothesis4}}
Let the notation and hypotheses be as in \cref{l13}. Then, we have 
    \[
        \sum_{\substack{N \leq n < 2N \\ n\equiv c \bmod q}} \mathbf{1}_{\mathcal{B}}(n) \ll_{f, \mathcal{H}, k, U_\ell} \frac{1}{q} \sum_{N \leq n < 2N}\mathbf{1}_{\mathcal{B}}(n)
    \]
for all $q<N^\theta.$ 
\end{Corollary}
\begin{proof}
    By \eqref{equ_no_congruence} and \eqref{equ_congruence}, we have
    \begin{align*}
        \sum_{N\leq n < 2N} \mathbf{1}_{\mathcal{B}}(n) = \lambda_\ell(U_\ell)N + O_{f,\mathcal{H}, k}(N^{1-\beta'})
    \end{align*}
    and
    \begin{align*}
        \sum\limits_{\substack{N\leq n  < 2N \\ n \equiv c \bmod q}} \mathbf{1}_{\mathcal{B}}(n) =  \lambda_\ell(U_\ell) \frac{N}{q} + O_{f, \mathcal{H}, k}(N^{1-\beta'}),
    \end{align*}
    and therefore
    \begin{align*}
        \sum_{\substack{N \leq n < 2N \\ n\equiv c \bmod q}} \mathbf{1}_{\mathcal{B}}(n) \ll_{f, \mathcal{H}, k, U_\ell} \frac{1}{q} \sum_{N \leq n < 2N}\mathbf{1}_{\mathcal{B}}(n).
    \end{align*}
    
\end{proof}

\section{Uniform distribution of values of polynomials with real exponents in arithmetic progressions along primes}\label{discr_primes}
In this section, we prove results that allow us to verify Hypotheses \ref{H1} and \ref{H3} of \cref{maynard} in Section \ref{lastsection}. To do this, we study the distribution of sequences of the form $(\{ F(p+h_1) \}, \ldots, \{ F(p+h_\ell) \})$ in $[0,1)^\ell$, where $p$ runs over the primes in an interval, or the primes in an interval lying in some arithmetic progression. \cref{expsumsoverprimes} and \cref{a(q)} are analogues of \cref{be} and \cref{corbdd} for exponential sums over the primes. Using these results, we prove \cref{Hypothesis 3}, which will help us verify Hypothesis \ref{H3}, and \cref{Hypothesis 1}, which will help us verify Hypothesis \ref{H1}. But first, we prove an easy lemma that is useful in the proof of \cref{expsumsoverprimes}.

\begin{Lemma}\label{supinterval}
    Let $\alpha > 0$, $W \geq 0$. Let $f : \N \rightarrow \R$ be such that $|f(n)| \leq 1$ for all $n \leq 2W$. If
    $$\left| \sum\limits_{N\leq n< 2N}f(n) \right|\leq C N^{\alpha} $$
    uniformly for $N \geq W$, then we have that
$$\sup\limits_U\sup\limits_V\left| \sum\limits_{\substack{N\leq n < 2N \\ U\leq n<V}}f(n) \right|\ll_\alpha C N^{\alpha} + W, $$ 
where the supremums are over all the positive real numbers.
\end{Lemma}
\begin{proof}
   Fix any $U, V \in \R$. Without loss of generality, we can assume that $N \leq U < V \leq 2N$. By using the triangle inequality, we get that
    \begin{equation*}
        \left| \sum\limits_{\substack{N\leq n < 2N \\ U\leq n<V}}f(n) \right|=\left| \sum\limits_{\substack{U\leq n<V}}f(n) \right|\leq \left| \sum\limits_{1\leq n<U}f(n) \right|+\left| \sum\limits_{1\leq n<V}f(n) \right|.
    \end{equation*}
    Let $k \geq 0$ be the smallest integer such that $U/2^k \leq 2W$. We have
    \begin{align*}
        \left| \sum\limits_{1\leq n<U}f(n) \right| &\leq \sum_{j=0}^{k-1} \left| \sum\limits_{U/2^{j+1} \leq n< U/2^j}f(n) \right| + \left| \sum\limits_{1 \leq n< U/2^k}f(n) \right|\\
        &\leq C U^\alpha \sum_{j=0}^{k-1} \frac{1}{(2^\alpha)^{j+1}}  + W\\
        &\ll_{\alpha} CN^\alpha + W.
    \end{align*}
    Similarly, $\left| \sum\limits_{1\leq n<V}f(n) \right| \ll_{\alpha} C N^\alpha + W$, so
    \begin{align*}
        \left| \sum\limits_{\substack{N\leq n < 2N \\ U\leq n<V}}f(n) \right| \ll_{\alpha} C N^\alpha + W.
    \end{align*}
\end{proof}

Next, we use the von Mangoldt function and partial summation to compare the sums
\[
    \left| \sum_{N\leq p <2N}e(\langle \mathbf{m}, \bar f(p)\rangle + P(p)) \right|\qquad \text{and} \qquad\left|\sum_{N \leq n < 2N}e(\langle \mathbf{m}, \bar f(n)\rangle + P(n)) \right|.
\]
This allows us to deduce the proof of \cref{expsumsoverprimes} using \cref{lem_BKM+_reformulated} and some properties of the von Mangoldt function (see \cite[Lemma 2.3]{Bergelson2014}). 

\begin{Proposition}\label{expsumsoverprimes}
    Let $k,r\in\N$ and fix some polynomial with non-integer exponents $f(x) = \sum_{j=1}^r d_j x^{\rho_j}$ with $0 < \rho_1 < \rho_2 < \ldots < \rho_r$ and $d_r \neq 0$. Let $\ell = \lfloor \rho_r \rfloor + 1$ and let $\eta$ be the same number as in Section \ref{discr_integers}. Let $\mathcal{H} = \{ h_1, \ldots, h_\ell \}$ be a set of integers with $h_1 < \ldots < h_\ell$, and define $\bar f(n) = (f(n + h_1), \ldots f(n + h_{\ell}))$. There exists a real number $\zeta>0$ depending only on $f$ and $k$ such that uniformly over all real polynomials $P$ of degree at most $k$, all $N\in\N$ and all vectors $\mathbf{m} = (m_1 \ldots m_{\ell}) \in \Z^{\ell} \setminus \{\bar{0}\}$ with $\max\limits_{1 \leq j \leq \ell}|m_j| \leq N^{\eta}$, 
    \begin{equation}\label{eqn_uni_dis_of_gen_poly_in_arith_prog_primes_1}
    \left| \sum_{N\leq p <2N}e(\langle \mathbf{m}, \bar f(p)\rangle + P(p)) \right| \ll_{f,\mathcal{H},k} N^{1-\zeta}.
    \end{equation}
\end{Proposition}
\begin{proof}
We follow the proof of \cite[Proposition 2.1]{Bergelson2014}. Let $S$ denote the left side of \eqref{eqn_uni_dis_of_gen_poly_in_arith_prog_primes_1}. By using partial summation, we have
   \begin{align*}
       S & =\left| \sum\limits_{N\leq p < 2N}e(\langle \textbf{m}, \bar f(p)\rangle + P(p)) \right| \\
       &= \left| \sum\limits_{N\leq p < 2N}\frac{\Lambda(p)}{\log p}e(\langle \textbf{m}, \bar f(p)\rangle + P(p)) \right| 
       \\
       & = \left| \sum\limits_{N\leq n < 2N}\frac{\Lambda(n)}{\log n}e(\langle \textbf{m}, \bar f(p)\rangle + P(n)) - \sum\limits_{N\leq p^k < 2N, k \geq 2}\frac{\Lambda(p)}{\log p}e(\langle \textbf{m}, \bar f(p)\rangle + P(p))\right| 
       \\
       &\leq\left| \sum\limits_{N\leq n < 2N} \frac{\Lambda(n)}{\log n}e(\langle \textbf{m}, \bar f(n)\rangle + P(n))\right| + O(N^{1/2}\log N)
       \\
       &\ll \frac{1}{\log N}\left| \sum\limits_{N\leq n < N_1}\Lambda(n)e(\langle \textbf{m}, \bar f(n)\rangle + P(n))\right| + O(N^{1/2} \log N)
   \end{align*}
  for the integer $N_1\in [N,2N]$ such that
  $$\left| \sum\limits_{N\leq n < N_1}\Lambda(n)e(\langle \textbf{m}, \bar f(n)\rangle + P(n))\right|$$
  is maximal. By Lemma \ref{supinterval} with $W = 0$, to get an upper bound as in \eqref{eqn_uni_dis_of_gen_poly_in_arith_prog_primes_1}, it is sufficient to find such an upper bound for the sum 
   \begin{align*}
       S_1=\sum\limits_{N \leq n < 2N}\Lambda(n)e(\langle \textbf{m}, \bar f(n)\rangle + P(n))
   \end{align*} 
   uniformly for $N \in \N$. To do this, we apply \cite[Lemma 2.3]{Bergelson2014}, with $F(x) = e(\langle \textbf{m}, \bar f(x)\rangle + P(x))\textbf{1}_{N \leq x < 2N}$, $|F(x)| \leq 1 = F_0 $, $U = \frac{1}{4}N^{1/5}$, $V = 4N^{1/3}$, and $Z$ a number in $\frac{1}{2} + \N$ at a minimal distance from $\frac{1}{4}N^{2/5}$. For all $N$ large enough (say $N \geq N_0$), we have $3 \leq U < V < Z < 2N$, $Z \geq 4U^2$, $2N \geq 64 Z^2 U$, $V^3 = 64N$, so the conditions of the lemma are satisfied. Then, by the conclusion of the lemma, we have
   \begin{equation*}
       \left| S_1\right|\ll 1+L \left(\log N\right)^8+K\log N
   \end{equation*}
for 
\begin{align*}
&K =\max_M \sum_{t=1}^{\infty}\tau_3(t) \left|\sum_{\substack{N \leq nt < 2N \\ Z<n\leq M}}e(\langle \textbf{m}, \bar f(nt)\rangle + P(nt)) \right|,\\
&L=\sup_b \sum_{t=1}^{\infty} \tau_4(t)\left|\sum_{\substack{ N \leq nt < 2N \\ U<n<V}}b(n) e(\langle \textbf{m}, \bar f(nt)\rangle + P(nt))\right|,\\
\end{align*}
where the supremum is taken over all arithmetic functions $b(n)$ satisfying $|b(n)|\leq \tau_3(n)$. Fix $\epsilon > 0$. Note that in the expression for $K$, the innermost sum is zero when $t \geq 2N/Z$, so we have
\begin{align*}
    K&=\max_{M} \sum_{t < \frac{2N}{Z}} \tau_3(t) \left|\sum_{\substack{\frac{N}{t} \leq n < \frac{2N}{t}\\Z<n\leq M}}e(\langle \textbf{m}, \bar f(nt)\rangle + P(nt)) \right|\\
    &\ll_\varepsilon N^\varepsilon \sum_{t < \frac{2N}{Z}} \max_{M} \left|\sum_{\substack{\frac{N}{t} \leq n < \frac{2N}{t}\\Z<n\leq M}}e(\langle \textbf{m}, \bar f(nt)\rangle + P(nt)) \right|,
\end{align*}
where we have used the trivial upper bound $\tau_3(t) \ll_\varepsilon t^\varepsilon \ll N^\varepsilon$.\\

As in the proof of \cref{be}, we can write $\langle \textbf{m}, \bar f(x)\rangle = f_1(x) + R(x)$, where $R(x) \ll_{f,\mathcal{H}} x^{\{\rho_r\} - 1 + \eta}$ and $f_1$ is of the form 
\begin{align*}
    f_1(x) = b x^{\rho_r - n_0} + \sum_{i = 1}^{r'} b_{i} x^{\rho_i'},
\end{align*}
where $n_0 \in \{0,\ldots, \ell - 1 \}$, $\rho_1' < \rho_2' < \ldots < \rho_{r'}' < \rho_r - n_0 - \eta$, $b=\left( \sum_{j=1}^\ell m_j h_j^{n_0} \right) d_r \binom{\rho_r}{n_0}$, $|b|,|b_i| \ll_{f, \mathcal{H}} N^\eta$ and $r' \leq \ell r$. By a variant of the argument in \cref{be}, we have
\begin{align*}
    \left|\sum_{\substack{\frac{N}{t} \leq n < \frac{2N}{t}\\Z<n\leq M}}e(\langle \textbf{m}, \bar f(nt)\rangle + P(nt)) \right| \ll_{f, \mathcal{H}} \left|\sum_{\substack{\frac{N}{t} \leq n < \frac{2N}{t}\\ Z < n \leq M}}e(f_1(nt) + P(nt))  \right| + \frac{N^{\frac{1 + \{ \rho_r \}}{2}}}{t}
\end{align*}
for any $M$. To get an upper bound on the sum on the right-hand side, we start by bounding
\begin{align*}
    S_2 := \left|\sum_{\frac{N}{t} \leq n < \frac{2N}{t}}e(\langle \textbf{m}, \bar f(nt)\rangle + P(nt))  \right|.
\end{align*}
As in the proof of \cref{be}, for some constants $0 < c_1 < c_2$ depending only on $f$ and $k$, one has
\begin{align}
\label{eqn_lfq_1}
    c_1 |b| N^{\rho_r - n_0}\left(\frac{N}{t}\right)^{-s} \leq \left|\frac{d^{s}}{d x^s} f_1(tx) \right| \leq c_2  |b| N^{\rho_r - n_0}\left(\frac{N}{t}\right)^{-s}
\end{align}
for any integers $0 \leq s \leq 5(\ell+1) + k + 2$, $1 \leq t < N^{4/5}$ and any real number $N/t \leq x \leq 2N/t$, as long as $N$ is large enough in terms of $f$ and $\mathcal{H}$. Therefore, by applying \cref{lem_BKM+_reformulated} to the function $x \mapsto f_1(xt)$ and with $q = 5(\ell + 1)$, we get that uniformly over the integers $1 \leq t \leq N^{4/5}$,
\begin{align*}
    \left|\sum_{\frac{N}{t} \leq n < \frac{2N}{t}}e(f_1(nt) + P(nt))  \right| &\ll_{f,\mathcal{H},k} \left(\frac{N}{t}\right)^{1-\frac{1}{2^k}} + \frac{N}{t} \left( \frac{(\log(N/t))^k}{|b|N^{\rho_r - n_0}} \right)^{\frac{1}{2^k}}\\
    &\qquad + \frac{N}{t} \left( \frac{|b|N^{\rho_r - n_0}}{(N/t)^{5(\ell + 1) + 2}}  \right)^{\frac{1}{2^{k + 5(\ell+1) + 2} - 2^{k + 1}}}\\
    &\ll_{f,\mathcal{H},k,\varepsilon} \left(\frac{N}{t}\right)^{1-\frac{1}{2^k}} + \frac{N^{1-\frac{\{\rho_r\}}{2^k} + \frac{\varepsilon}{2}}}{t} + \frac{N^{1 - \frac{1 - \{\rho_r\}}{2^{k + 5(\ell+1) + 2} - 2^{k + 1}}}}{t}\\
    &\ll_{f,\mathcal{H},k,\varepsilon} \left( \frac{N}{t} \right)^{1 - \alpha'},
\end{align*}
where $\alpha' = \min \left\{ \frac{1}{2^k}, \frac{\{\rho_r\}}{2^k} - \varepsilon,  \frac{1 - \{\rho_r\}}{2^{k + 5(\ell+1) + 2} - 2^{k + 1}} - \varepsilon \right\}$. For any fixed $t$ large enough in terms of $f$ and $\mathcal{H}$, this estimate holds as long as $N/t \geq N^{1/5}$, so, by applying \cref{supinterval}  with $W = N^{1/5}$, we get
\begin{align*}
    \max_M \left|\sum_{\substack{\frac{N}{t} \leq n < \frac{2N}{t}\\Z<n\leq M}}e(\langle \textbf{m}, \bar f(nt)\rangle + P(nt)) \right| \ll_{f,\mathcal{H},k,\varepsilon} \left( \frac{N}{t} \right)^{1 - \alpha'} + N^{1/5}
\end{align*}
and so
\begin{align*}
    K \ll_{f,\mathcal{H},k,\varepsilon} N^{1-\zeta_1},
\end{align*}
where
\begin{align*}
    \zeta_1 = \min \left\{ \frac{3}{5}\alpha', \frac{1}{5} \right\} - \varepsilon.\\
\end{align*}

Next, we find an upper bound for $L$. Let 
\begin{align*}
    S_3 &= \sum_{t=1}^\infty \tau_4(t) \left| \sum_{\substack{N\leq nt<2N\\U<n<V}}b(n)e(\langle \textbf{m}, \bar f(nt)\rangle+P(nt)) \right|
\end{align*}
for a fixed arithmetic function $b(n)$ satifsying $|b(n)| \leq \tau_3(n)$.
Again, since $\tau_4(n) \ll_\varepsilon n^\varepsilon$, we have
\begin{align*}
    S_3 &\ll_\varepsilon N^{\epsilon} \sum_{t =1}^{\infty} \left| \sum_{\substack{U<n<V\\ \frac{N}{t} \leq n < \frac{2N}{t}}}b(n)e(\langle \textbf{m}, \bar f(nt)\rangle+P(nt)) \right|.
\end{align*}
The innermost sum is zero when $t \leq N/V$ or $t \geq 2N/U$, so, in fact, we have
\begin{align*}
    S_3 \ll_\varepsilon N^{\epsilon} \sum_{\frac{N}{V} < t < \frac{2N}{U}} \left| \sum_{\substack{U<n<V \\ \frac{N}{t} \leq n < \frac{2N}{t}}}b(n)e(\langle \textbf{m}, \bar f(nt)\rangle+P(nt)) \right|.
\end{align*}
By replacing the lower bound $N/V$ with some integer between $N/(2V)$ and $N/V$, the sum over $t$ can be decomposed into $\ll \log N$ sums over $t$ in dyadic intervals $[X,2X)$. Let 
\begin{align*}
S_4 = \sum_{X \leq t < 2X} \left| \sum_{\substack{U<n<V \\ \frac{N}{t} \leq n < \frac{2N}{t}}}b(n)e(\langle \textbf{m}, \bar f(nt)\rangle+P(nt)) \right|
\end{align*}
be one of those sums. Then, by the Cauchy-Schwarz inequality, we have
\begin{align*}
    |S_4|^2& = \left(\sum_{X \leqslant t < 2X} 1 \cdot \left| \sum_{\substack{U<n<V \\ \frac{N}{t} \leq n < \frac{2N}{t}}}b(n)e(\langle \textbf{m}, \bar f(nt)\rangle+P(nt)) \right|\right)^2 
    \\ 
   & \leqslant \left(\sum_{X \leqslant t< 2X} 1^2 \right) \left(\sum_{X \leqslant t < 2X}\left|  \sum_{\substack{U<n<V \\ \frac{N}{t} \leq n < \frac{2N}{t}}}b(n)e(\langle \textbf{m}, \bar f(nt)\rangle+P(nt)) \right|^2 \right) \\
   & \leqslant X \sum_{X \leqslant t < 2X}\left|  \sum_{\substack{U<n<V \\ \frac{N}{t} \leq n < \frac{2N}{t}}}b(n)e(\langle \textbf{m}, \bar f(nt)\rangle+P(nt)) \right|^2 \\
   &= X \sum_{X \leqslant t < 2X} \left(\sum_{\substack{U<n<V \\ \frac{N}{t} \leq n < \frac{2N}{t}}}b(n)e(\langle \textbf{m}, \bar f(nt)\rangle+P(nt)) \right) 
   \\
   &\qquad \qquad \qquad \qquad \qquad \qquad\times\left(\sum_{\substack{U<n<V \\ \frac{N}{t} \leq n < \frac{2N}{t}}}\overline{b(n)}e(-\langle \textbf{m}, \bar f(nt)\rangle-P(nt)) \right) \\
   & \ll X \sum_{X \leqslant t < 2X} \sum_{\substack{U<n<V \\ \frac{N}{t} \leq n < \frac{2N}{t}}}|b(n)|^2
   \\
   &
   + X \left| \sum_{X \leqslant t < 2X} \sum_{\substack{A<n_1,n_2<B \\ n_1 \neq n_2}}b(n_1)\overline{b(n_2)}e(\langle \textbf{m}, \bar f(n_1t)-\bar f(n_2t)\rangle+P(n_1t)-P(n_2t))\right|,
\end{align*}
where $A=\max\left\{U, \lceil N/t \rceil - 1\right\}$ and $B=\min\left\{V, 2N/t\right\}$. By changing the order of summation and using the bound $|b(n)|\leq \tau_3(n)\ll_\varepsilon n^{\varepsilon}\ll_{\varepsilon} N^{\varepsilon}$ for $n \leq 2N$, we get 
\begin{equation}\label{equ_S4}
    |S_4|^2\ll_\varepsilon N^{\epsilon}X\left(N+\sum\limits_{\substack{A'\leq n_1,n_2<B' \\ n_1 \neq n_2}}\left|\sum\limits_{\substack{X\leq t<2X \\ \frac{N}{n_1} \leq  t <  \frac{2N}{n_2}}} e(\langle \textbf{m}, \bar f(n_1t)-\bar f(n_2t)\rangle+P(n_1t)-P(n_2t))\right|\right),
\end{equation}
with $A' = \max \{ U, \lceil N/(2X) \rceil - 1\}$ and $B' = \min\left\{V, 2N/X\right\}$.
For fixed $n_1$ and $n_2 \neq n_1$ in that range, we have that $Q(t)=P(n_1t)-P(n_2t)$ is a polynomial of degree at most $k = \deg P(x)$. Also, we can write 
\begin{align*}
    \langle\textbf{m}, \bar f(n_1t)-\bar f(n_2t)\rangle = f_1(n_1t) - f_1(n_2t) + R(n_1t) - R(n_2t),
\end{align*}
where $f_1$ and $R$ are the same functions as above. As before,
\begin{align*}
    &\left|\sum\limits_{\substack{X\leq t<2X \\ \frac{N}{n_1} \leq  t <  \frac{2N}{n_2}}} e(\langle \textbf{m}, \bar f(n_1t)-\bar f(n_2t)\rangle+P(n_1t)-P(n_2t))\right|\\
    &\qquad\ll_{f,\mathcal{H}} \left|\sum\limits_{\substack{X\leq t<2X \\ \frac{N}{n_1} \leq  t <  \frac{2N}{n_2} }} e(f_1(n_1t) - f_1(n_2t)+P(n_1t)-P(n_2t))\right| + n_2^{\frac{\{\rho_r\} - 1}{2}}X^{\frac{1+\{\rho_r\}}{2}}.
\end{align*}
Now, there are some constants $c_1', c_2'$ depending only on $f$ and $k$ such that 

\begin{align*}
     \left|\frac{\d^{s}}{\d t^s} (f_1(n_1t) - f_1(n_2t)) \right| 
     &\geq c_1 |b(n_1^{\rho_r - n_0} - n_2^{\rho_r - n_0})| X^{\rho_r - n_0} X^{-s}\\
     &\geq c_1'  \left|b\frac{(n_1 - n_2)}{n_1}\right| N^{\rho_r - n_0}X^{-s},\\
\end{align*}

\begin{align*}
    \left|\frac{\d^{s}}{\d t^s} (f_1(n_1t) - f_1(n_2t)) \right| 
    &\leq c_2  |b(n_1^{\rho_r - n_0} - n_2^{\rho_r - n_0})| X^{\rho_r - n_0}X^{-s}\\
    &\leq c_2'  \left|b\frac{(n_1 - n_2)}{n_1}\right| N^{\rho_r - n_0}X^{-s}\\
\end{align*}
for all $0 \leq s \leq 3(\ell+1)+k+2$ as long as $N$ is large enough in terms of $f$ and $\mathcal{H}$, so by applying \cref{lem_BKM+_reformulated} to the function $t \mapsto f_1(n_1t) - f_1(n_2t)$ and with $q = 3(\ell+1)$, we get

\begin{align*}
   & \left|\sum\limits_{X\leq t<2X} e(f_1(n_1t) - f_1(n_2t)+P(n_1t)-P(n_2t))\right|  \\ 
&\quad
  \ll_{f,\mathcal{H},k}  X^{1-\frac{1}{2^{k}}}+ X\left( \frac{n_1(\log X)^k}{|b(n_1 - n_2)| N^{\rho_r - n_0}}\right)^{1/2^k} \\
    & \qquad + X \left( \frac{|b(n_1 - n_2)| N^{\rho_r - n_0}}{n_1X^{3(\ell+1)+2}}\right)^{\frac{1}{2^{k+3\ell+5} - 2^{k+1}}}\\
    &\quad\ll_{f,\mathcal{H},k,\varepsilon} X^{1-\frac{1}{2^{k}}} + X^{1+\epsilon}\left( \frac{n_1}{|n_1 - n_2| N^{\rho_r - n_0}}\right)^{1/2^k} + X N^{ - \frac{1-\{\rho_r\}}{2^{k+3\ell+5} - 2^{k+1}}}.
\end{align*}
Here we used the fact that $X\geq N/2V \gg N^{2/3}\geq N^{1/3}$ and $1 \ll_{f, \mathcal{H}} b\ll_{f,\mathcal{H}} N^\eta < N.$ In fact, the result holds for $X \geq N^{1/3}$, so by applying \cref{supinterval} with $W = N^{1/3}$, we deduce that
\begin{align*}
    &\left|\sum\limits_{\substack{X\leq t<2X \\ \frac{N}{n_1} \leq  t <  \frac{2N}{n_2}}} e(f_1(n_1t) - f_1(n_2t)+P(n_1t)-P(n_2t))\right|\\
    &\qquad\ll_{f,\mathcal{H},k,\epsilon}X^{1-\frac{1}{2^{k}}} + X^{1+\epsilon}\left( \frac{n_1}{|n_1 - n_2| N^{\rho_r - n_0}}\right)^{1/2^k} + XN^{ - \frac{1-\{\rho_r\}}{2^{k+3\ell+5} - 2^{k+1}}}+N^{\frac{1}{3}}.
\end{align*}
Notice that
\begin{align*}
    \sum\limits_{\substack{A'\leq n_1,n_2<B' \\ n_1 \neq n_2}}\left( \frac{n_1}{|n_1 - n_2|}\right)^{1/2^k} &\ll_k \left(\sum\limits_{A'\leq n_1<B'}n_1^{\frac{1}{2^k}}\right)\left(\sum\limits_{1\leq n_3\leq B'}n_3^{-\frac{1}{2^k}}\right)\\
    &\ll_k (B')^2,
\end{align*}
so by inserting these upper bounds into \eqref{equ_S4}, we get

\begin{align*}
    |S_4|^2 &\ll_{f,\mathcal{H}, k , \varepsilon} N^{2\varepsilon} (XB')^2 N^{-\frac{\{\rho_r\}}{2^k}}+N^{\epsilon}(B'X)^2X^{-\frac{1}{2^{k}}}\\
    &\qquad + (B'X)^2 N^{\varepsilon} N^{ - \frac{1-\{\rho_r\}}{2^{k+3\ell+5} - 2^{k+1}}} + XB'N^{\frac{1+\{\rho_r\}}{2}+\epsilon}+XN^{1+\epsilon} + X (B')^2N^{\frac{1}{3} + \varepsilon}\\
    &\ll_{f,k,\mathcal{H}, \epsilon} N^{2-\frac{\{\rho_r\}}{2^k}+2\epsilon}+N^{2 - \frac{1}{3\cdot2^{k-1}} +\epsilon}+N^{2 - \frac{1- \{\rho_r\}}{2^{k+3\ell+5} - 2^{k+1}}+\epsilon}+N^{\frac{3+\{\rho_r\}}{2}+\epsilon}+N^{\frac{5}{3}+\epsilon}\\
     &\ll_{f,k,\mathcal{H},\epsilon} N^{\beta}
\end{align*}
where 
\begin{align*}   
\beta = \max\left\{ 2 - \frac{\{\rho_r\}}{2^{k}} + 2\varepsilon, 2-\frac{1 - \{ \rho_r \}}{2^{k+3\ell+5} - 2^{k+1}} + \epsilon,\frac{3+\{ \rho_r \}}{2} + \varepsilon, \frac{5}{3} + \varepsilon\right\}.
\end{align*}

Here, we used the fact that $B'X\ll N$, $X\gg \frac{N}{V}\gg N^{\frac{2}{3}}$ and $X \ll \frac{N}{U} \ll N^{4/5}$. Note that $\beta < 2$ if $\varepsilon$ is small enough.
Now, summing over all subintervals, we get 
\begin{equation*}
    L\ll_{f,\mathcal{H},k,\epsilon}N^{1-\zeta_2}
\end{equation*}
where $\zeta_2=1-\frac{\beta}{2}-\epsilon.$ For $\epsilon > 0$ small enough and $\zeta:=\min\{\zeta_1,\zeta_2\}$ we get the desired result.\\
\end{proof}

Again, as in \cref{corbdd}, we get the same bound if we restrict the sum to primes in some arithmetic progression.

\begin{Corollary}\label{a(q)}
Under the hypotheses of \cref{expsumsoverprimes}, uniformly over all real polynomials $P$ of degree at most $k$, all $N\in\N$, all $\mathbf{m} = (m_1, \ldots, m_\ell) \in \Z^{k} \setminus\{\bar 0\}$ with $\max\limits_{1 \leq j \leq \ell} |m_j| \leq N^{\eta}$, all $q\in\N$, and all $a\in\{0,1,\ldots,q-1\}$ one has
    \begin{equation}
    \left| \sum_{\substack{N\leq p < 2N\\ p\equiv a\bmod q}}e(\langle \mathbf{m}, \bar f(p)\rangle + P(p)) \right| \ll_{f,\mathcal{H},k} N^{1-\zeta },
    \end{equation}
    where $\zeta$ and the implied constant are the same as in Proposition \ref{expsumsoverprimes}.
\end{Corollary}

\begin{proof}
As in the proof of Corollary \ref{corbdd}, we use the fact that
\[
\mathbf{1}_{q\Z+a}(n)=\frac{1}{q}\sum_{0 \leq t \leq q-1} e\left( \frac{(n-a)t}{q}\right).
\]
Therefore, we have
\begin{align*}
\left| \sum_{\substack{N\leq p < 2N \\ p\equiv a\bmod q}} e(\langle \mathbf{m}, \bar f(p)\rangle+P(p)) \right|
&=
\left| \sum_{N\leq p < 2N } e(\langle \mathbf{m}, \bar f(p)\rangle+P(p)) 1_{q\Z+a}(p) \right|
\\
&=
\left| \sum_{N\leq p < 2N } e(\langle \mathbf{m}, \bar f(p)\rangle+P(p)) \left(\frac{1}{q}\sum_{0 \leq t \leq q-1} e\left( \frac{(p-a)t}{q}\right)\right) \right|
\\
&\leq \frac{1}{q}\sum_{0 \leq t \leq q-1}
\left| \sum_{N\leq p < 2N } e\left(\langle \mathbf{m}, \bar f(p)\rangle+P(p)+\frac{(p-a)t}{q}\right) \right|
\\
&= \frac{1}{q}\sum_{0 \leq t \leq q-1}\left| \sum_{N\leq p < 2N } e\left(\langle \mathbf{m}, \bar f(p)\rangle+P(p)+\frac{tp}{q}\right) \right|.
\end{align*}
The claim now follows from \cref{expsumsoverprimes}.\\
\end{proof}

Now, we prove a version of the Bombieri-Vinogradov theorem (\cref{Bombieri}) for the sets of our interest (cf. \cite[Theorem 2]{shubin2020fractional}). We start by following the steps of the proof of \cref{l13} (\cref{expsumsoverprimes} and \cref{a(q)} play a role analogous to \cref{be} and \cref{corbdd}).

\begin{Lemma}\label{Hypothesis 3}
    Let $k,r \in \N$, and fix some polynomial with non-integer exponents $f(x) = \sum_{j=1}^r d_j x^{\rho_j}$ with $0<\rho_1 < \rho_2 < \ldots  < \rho_r$ and $d_r \neq 0$. Then, there exists some $\theta > 0$ depending only on $f(x)$ and $k$ such that the following holds:
    
    Let $h_1 < h_2 < \ldots < h_\ell$ be integers. Let $P$ be a real polynomial of degree at most k. Let $U_\ell = \prod_{i=1}^\ell [u_i, v_i)$ with $0 \leq u_i < v_i < 1$ for all $i$, and define
    \[\mathcal{B}:=\{n\in\N: (\{f(n+h_1)+P(n+h_1)\}, \ldots, \{f(n+h_\ell)+P(n+h_\ell)\})\in U_\ell\}.\]
    Fix a $D > 0$. Then, for all $N$ large enough depending on all that precedes, we have
    \begin{equation}\label{eqn_to_prove}
    \sum_{q<N^\theta} \max_{(q,a)=1}\left|\sum_{\substack{N\le p < 2N\\p\equiv c\bmod q }}\textbf{1}_{\mathcal{B}}(p)-\frac{1}{\varphi(q)}\sum_{\substack{N\le p< 2N }}\textbf{1}_{\mathcal{B}}(p)\right| \ll_{f,\mathcal{H},D,k,U_\ell} \frac{1}{(\log N)^D} \sum\limits_{\substack{N\le p< 2N }}\textbf{1}_{\mathcal{B}}(p).
    \end{equation}
    
\end{Lemma}
\begin{proof}
Define
    \[ \bar f (x) :=  (f(x +h_1), \ldots, f(x +h_\ell))\]
    and
    \[ \bar P (x) :=  (P(x + h_1), \ldots, P(x  + h_\ell)).\]
By the \Erdos-\Turan-Koksma inequality (\cref{discbd}), the discrepancy mod 1 of 
$(\bar{f}(p)+\bar{P}(p))_{p \in [N,2N)}$ verifies
\begin{align*}
     D\left((\bar{f}(p)+\bar{P}(p))_{p \in [N,2N)\cap \mathbb{P}}\right) \ll_\ell \frac{1}{H} + \sum_{\substack{\mathbf{m} \in \Z^k \setminus \{\bar{0}\} \\ \lVert \mathbf{m} \rVert_\infty \leq H}} \frac{1}{r(\mathbf{m})} \frac{1}{\pi(N,2N)}\left| \sum_{N \leq p < 2N} e(\langle \mathbf{m}, \bar{f}(p)+\bar{P}(p) \rangle) \right| 
\end{align*}
for any $H \in \N.$ Now, observe that
\begin{align*}
    \left| \sum_{N \leq p < 2N} e(\langle \mathbf{m}, \bar{f}(p)+\bar{P}(p) \rangle) \right|& =  \left| \sum_{N \leq p < 2N} e(\langle \mathbf{m}, \bar f(p)\rangle  + \langle \mathbf{m}, \bar P(p)\rangle) \right|\\
    &\ll_{f,\mathcal{H},k} N^{1 - \zeta}
\end{align*}
  by Proposition \ref{expsumsoverprimes}. Moreover, by the prime number theorem, $\frac{1}{\pi(N, 2N)} \ll \frac{\log N}{N}$ and, as in the proof of Lemma \ref{l13}, we also have 
\[
    \sum_{\substack{m \in \Z^k \setminus \{\bar{0}\} \\ \lVert m \rVert_\infty \leq H}}\frac{1}{r(m)} \ll (\log H)^\ell.
    \]
Thus,
\[
D\left((\bar{f}(p)+\bar{P}(p))_{p \in [N,2N) \cap \mathbb{P}}\right)\ll_{f,\mathcal{H},k} \frac{1}{H}+(\log H)^\ell\frac {N^{ - \zeta}}{\log N}.
\]
Taking $H=\lfloor N^\eta \rfloor,$ we get
\[
D\left((\bar{f}(p)+\bar{P}(p))_{p \in [N,2N)\cap \mathbb{P}}\right)\ll_{f,\mathcal{H},k,\gamma} N^{-\gamma}
\]
for all $0<\gamma<\min\{\eta,\zeta\}.$ For concreteness, let $\gamma = \frac{\min\{\eta,\zeta\}}{2}$.
It follows that 
\begin{align*}
\left|\frac{1}{\pi(N, 2N)}\sum_{\substack{N\le p < 2N }}\textbf{1}_{\mathcal{B}}(p)-\lambda_\ell(U_\ell)\right|\leq D\left((\bar{f}(p)+\bar{P}(p))_{p \in [N,2N) \cap \mathbb{P}}\right)\ll_{f,\mathcal{H},k} N^{-\gamma},
\end{align*}
which implies that
\[\sum_{\substack{N\le p < 2N}}\textbf{1}_{\mathcal{B}}(p)=\pi(N, 2N)\lambda_\ell(U_\ell)+O_{f,\mathcal{H},k}(N^{1-\gamma}).\]
Now, consider the sequence $\left(\bar{f}(p)+\bar{P}(p)\right)_{\substack{{p \in [N,2N)\cap \mathbb{P}} \\ p \equiv c\bmod{q}}}$ for integers $0 \leq a < q$. Again, as in the proof of \cref{l13}, using \cref{a(q)}, we get
\[D\left(\left(\bar{f}(p)+\bar{P}(p)\right)_{\substack{{p \in [N,2N)\cap \mathbb{P}} \\ p \equiv c \bmod{q}}}\right)\ll_{f,\mathcal{H},k} N^{-\gamma}.\]
Therefore, we have
\[
\sum_{\substack{N\le p < 2N\\ p \equiv c \bmod q }}\textbf{1}_{\mathcal{B}}(p)= \pi(N , 2N ;c,q)\lambda_\ell(U_\ell)+O_{f,\mathcal{H},k}(N^{1-\gamma}).
\]
Now, the left side of \eqref{eqn_to_prove} is of the form
\begin{equation*}\label{Sum4}
S:=\sum_{q<N^\theta}\max_{(q,c)=1}\left| \lambda_\ell(U_\ell) \pi(N,2N;c,q)-\lambda_\ell(U_\ell)\frac{\pi(N,2N)}{\varphi(q)}+O_{f,\mathcal{H},k}(N^{1-\gamma})\right|
\end{equation*}
The triangle inequality together with \cref{Bombieri} yield
\begin{align*}   
 S &\leq \lambda_\ell(U_\ell)\sum_{q<N^\theta}\max_{(q,c)=1}\left|\pi(2N;c,q)-\frac{\pi(2N)}{\varphi(q)}\right|\\
 &\quad +  \lambda_\ell(U_\ell)\sum_{q<N^\theta}\max_{(q,c)=1}\left|\pi(N;c,q) - \frac{\pi(N)}{\varphi(q)}\right|+O_{f,\mathcal{H},k}(N^{1-\gamma+\theta})\\
 &\ll_{D,U_\ell}\frac{N}{(\log N)^{D+1}}+O_{f,\mathcal{H},k}(N^{1-\gamma+\theta})
\end{align*}
for all $N$ large enough. But since
\[ \frac{N}{(\log N)^{D+1}}\ll\frac{\pi(N,2N)}{(\log N)^D}\ll_{f,\mathcal{H},k,U_\ell}\frac{1}{(\log N)^D} \sum_{\substack{N \le p < 2N }}\textbf{1}_{\mathcal{B}}(p),\]
by choosing, say, $\theta = \frac{\gamma}{2}$, we get the desired result.\\
\end{proof}

Finally, we show a result that will help us obtain Hypothesis \ref{H1} in Section \ref{lastsection}. It will follow from the estimates
\[\sum_{ \substack{N \leq n < 2N \\ n \in \mathbb{P}}} \mathbf{1}_{\mathcal{B}}(n) = \lambda_\ell(U_\ell)  \pi(N, 2N) + O_{f, \mathcal{H},k}(N^{1-\gamma})\]
and
\begin{equation*}
    \sum_{N\leq n < 2N} \mathbf{1}_{\mathcal{B}}(n) = \lambda_\ell(U_\ell)N + O_{f,\mathcal{H}, k}(N^{1-\beta'})
\end{equation*}
obtained in the proofs of \cref{Hypothesis 3} and \cref{l13}, respectively. 

    \begin{Lemma}\label{Hypothesis 1}
  Under the hypotheses of \cref{Hypothesis 3}, for all $0 < \delta < 1$ and $N$ large enough in terms of all that precedes, we have
    \begin{align*}
         \sum_{ \substack{N \leq n < 2N \\ n \in \mathbb{P}}} \mathbf{1}_{\mathcal{B}}(n) \geq \frac{\delta}{\log N} \sum_{ N \leq n < 2N } \mathbf{1}_{\mathcal{B}}(n).
    \end{align*}
    \end{Lemma}

    \begin{proof}
        As in the proof of \cref{Hypothesis 3}, for $i = 1, \ldots, \ell$, we can write
          $$ \sum_{ \substack{N \leq n < 2N \\ n\in \mathbb{P}}} \mathbf{1}_{\mathcal{B}}(n) = \lambda_\ell(U_\ell)  \pi(N, 2N) + O_{f, \mathcal{H},k}(N^{1-\sigma'})$$
          for some $\sigma' >0$, so, by the prime number theorem,
        \begin{equation*}
            \sum_{ \substack{N \leq n < 2N \\ n \in \mathbb{P}}} \mathbf{1}_{\mathcal{B}}(n) \sim \lambda_\ell(U_\ell) \frac{N}{\log N}.
        \end{equation*}
        As in the proof of \cref{l13}, we have
          \[\sum_{N\le n\le 2N}\textbf{1}_{\mathcal{B}}(n)=\lambda_\ell(U_\ell)N+O_{f,\mathcal{H},k}(N^{1-\sigma})\]
          for some $\sigma > 0$, so
          \begin{align*}
              \frac{1}{\log N}\sum_{N\le n\le 2N}\textbf{1}_{\mathcal{B}}(n) \sim \lambda_\ell(U_\ell) \frac{N}{\log N}.
          \end{align*}
          Therefore, 
          \begin{align*}
              \lim_{N \to \infty}\frac{ \sum\limits_{ \substack{N \leq n < 2N \\ n \in \mathbb{P}}} \mathbf{1}_{\mathcal{B}}(n)}{\frac{1}{\log N} \sum\limits_{ N \leq n < 2N } \mathbf{1}_{\mathcal{B}}(n)} = 1.
          \end{align*}
          In particular, for all $N$ large enough, the ratio on the left side is bigger than $\delta$.\\
    \end{proof}
\section{Proof of the main theorem}\label{lastsection}
In this section, we prove \cref{main_result}. Fix $F$ as in the statement of \cref{main_result}. We show that for any integer $m$, there is some integer $d \geq m$ as well as integers $h_1 < \ldots < h_d$ such that for any $N$ large enough, there exists some $n$ in the set
\begin{equation}\label{good_set}
    \{ n \in \N : \{F(n+h_j)\} \in U \text{ for } 1 \leq j \leq d \} \cap [N,2N)
\end{equation}
with the property that at least $m$ elements of the set $\{ n + h_1, \ldots, n + h_d \}$ are consecutive primes. An important ingredient in establishing the main results of Sections \ref{discr_integers} and \ref{discr_primes} was that if $\rho$ is the largest exponent appearing in the expression for $F$ and $\ell = \lfloor \rho \rfloor + 1$, then the sequence $(\{F(n+h_1)\}, \ldots, \{F(n+h_\ell)\})$ is uniformly distributed in $[0,1)^\ell$. This, however, does not hold if $\ell$ is replaced by any integer $d > \ell$, because, in that case, we can find a $\Z$-linear combination of $F(n+h_1), \ldots, F(n+h_d)$ that goes to zero as $n$ goes to infinity. In \cref{prop_relations}, we show that we can choose $h_1, \ldots, h_d$ in such a way that for $\ell + 1 \leq j \leq d$, the function $F(n+h_j)$ can be written as a $\Z$-linear combination of $F(n+h_1), \ldots, F(n+h_\ell)$, up to a small error term. This is the key ingredient of our argument, as, for this choice of admissible set, the value of $(\{F(n+h_1)\}, \ldots, \{F(n+h_d)\})$, which is not uniformly distributed in $[0,1)^d$, is completely determined by the value of the vector $(\{F(n+h_1)\}, \ldots, \{F(n+h_\ell)\})$, which is uniformly distributed in $[0,1)^\ell$.
\begin{Remark}\label{rem_notinz}
    This is the reason why we require that the largest exponent of $F$ is not an integer, as opposed to some exponent not being an integer. Indeed, if $F$ is a function of the form $ F(x) = \sum_{i=1}^r d_i x^{\rho_i}$ where $d_1,\ldots,d_r\in\R$, $0 \leq \rho_1 < \ldots < \rho_r,$ $\rho_i \in \mathbb{R}$ for $ 1 \leq i \leq r$ and $i_0$ is the largest integer such that $\rho_{i_0}$ is not an integer and $d_{i_0} \neq 0$, then similar arguments to those of Sections \ref{discr_integers} and \ref{discr_primes} would imply that, for $\ell_0 = \lfloor \rho_{i_0} \rfloor + 1$, the vector $(\{F(x + h_1) \}, \ldots, \{F(x + h_{\ell_0}) \})$ is uniformly distributed in $[0,1)^{\ell_0}$ for every choice of distinct $h_1, \ldots, h_{\ell_0}$, but our argument would fail, as for every integer $h \not \in \{ h_1, \ldots, h_{\ell_0} \}$, we would not be able to express $F(x+h)$ as a $\Z$-linear (or even $\Q$-linear) combination of the $F(x+h_1), \ldots, F(x+h_{\ell_0})$.
\end{Remark}

\begin{Proposition}\label{prop_relations}
    Let $\ell \in \N$. Then, for every $d \geq \ell$, there exists an admissible set $\mathcal{H} = \{ h_1, \ldots, h_d \}$ of $d$ integers such that the following holds:\\
    Let $F(x) = \sum_{i=1}^r d_i x^{\rho_i}$ with $0 \leq \rho_1 < \ldots < \rho_r$ be a polynomial with real exponents, with $\rho_r \not\in \Z, d_r \neq 0$ and $\ell = \lfloor \rho_r \rfloor + 1$. Then, for every $\ell + 1 \leq s \leq d$, there is some vector $(b_{s,1}, \ldots, b_{s,\ell}) \in \Z^\ell$ such that
    \begin{equation*}
        F(x + h_s) = \sum_{j = 1}^\ell b_{s,j} F(x + h_j) + O_{F, \mathcal{H}}(x^{\{ \rho_r \} - 1}).
    \end{equation*}
\end{Proposition}
\begin{proof}
    Again, for $h \in \Q$, write 
    $$\nu(h) = \begin{pmatrix}
        1\\
        h\\
        \vdots\\
        h^{\ell - 1}
    \end{pmatrix} \in \Q^{\ell}.$$
    Let 
    $$h_j = (j-1) \prod_{p \leq d} p$$
    for $j = 1, \ldots, \ell$, so that $0 = h_1 < h_2 < \ldots < h_\ell$ and $\{ h_1, \ldots, h_\ell \}$ is admissible. Then, the vectors $\nu(h_1), \ldots, \nu(h_\ell)$ are all linearly independent, by the same argument as in Proposition \ref{prop_lin_comb}. Therefore, $\nu(h_1), \ldots, \nu(h_\ell)$ form a basis of $\Q^\ell$. In particular, for $k = 1, \ldots, \ell$, if $e_k \in \Q^\ell$ is the vector with a 1 in the $k^{th}$ position and $0$'s in all the other positions, we can write
    \begin{align*}
        e_k = \sum\limits_{j=1}^\ell\frac{p_{k,j}}{q_{k,j}}\nu(h_j), 
    \end{align*}
    where $p_{k,j}\in\Z$, $q_{k,j}\in \N$ and $(p_{k,j},q_{k,j})=1$. Note that $e_1 = \nu(0) = \nu(h_1)$, so $q_{1,1} = 1$ and $p_{1,j} = 0$ for $j \geq 2$.
    
    Denote $q:=\prod\limits_{k,j=1}^\ell q_{k,j}$. Then, for any integer $t$, we have
    \[\nu(qt)=\sum\limits_{k=1}^\ell(qt)^{k-1}e_k=\sum\limits_{k=1}^\ell t^{k-1}q^{k-1}\left(\sum\limits_{j=1}^\ell\frac{p_{k,j}}{q_{k,j}}\nu(h_j)\right)=\sum\limits_{j=1}^\ell\left(\sum\limits_{k=1}^\ell t^{k-1}p_{k,j}\frac{q^{k-1}}{q_{k,j}}\right)\nu(h_j).\]
But $q_{1,1}=1,p_{1,j}=0,$ and $q_{k,j}|q^{k-1}$ for $k>1,$ so $\sum\limits_{k=1}^\ell t^{k-1}p_{k,j}\frac{q^{k-1}}{q_{k,j}}\in\Z$. Hence, by defining, say, 
    $$ h_s =  qs  \prod_{p \leq d} p $$
    for $\ell < s \leq d$, we have that $\{ h_1, \ldots, h_d \}$ is admissible, $h_1 < h_2 < \ldots < h_d$, and we can write 
    \begin{align*}
        \nu(h_s) = \sum_{j=1}^{\ell} b_{s,j} \nu(h_j)
    \end{align*}
    for 
    $$b_{s,j} = \sum\limits_{k=1}^\ell \left(s  \prod_{p \leq d} p\right)^{k-1}p_{k,j}\frac{q^{k-1}}{q_{k,j}} \in \Z.$$
    
    Now, as in the proof of Proposition \ref{prop_lin_comb}, by Taylor's theorem, for $j = 1, 2, \ldots, d$, and $x \geq h_d$, $x \geq 1$, we can write
    \begin{equation}\label{taylor expansion} 
        F(x + h_j) = \sum_{i = 1}^r d_i   \sum_{n=0}^{\ell - 1} \binom{\rho_i}{n} h_j^n x^{\rho_i - n}  + O_{F, \mathcal{H}} \left(x^{\{ \rho_r \} - 1} \right). 
    \end{equation}
    In particular, we have that for $s = \ell + 1, \ell + 2, \ldots, d$,
    \begin{align*}
        F(x + h_s) &= \sum_{i = 1}^r d_i   \sum_{n=0}^{\ell - 1} \binom{\rho_i}{n} h_s^n x^{\rho_i - n}  + O_{F,\mathcal{H}} \left(x^{\{ \rho_r \} - 1} \right)\\
        &= \sum_{i = 1}^r d_i   \sum_{n=0}^{\ell - 1} \binom{\rho_i}{n} \left(\sum_{j = 1}^\ell b_{s,j} h_j^n \right) x^{\rho_i - n}  + O_{F,\mathcal{H}} \left(x^{\{ \rho_r \} - 1} \right)\\
        &= \sum_{j = 1}^\ell b_{s,j} \sum_{i = 1}^r d_i   \sum_{n=0}^{\ell - 1} \binom{\rho_i}{n} h_j^n x^{\rho_i - n}  + O_{F,\mathcal{H}} \left(x^{\{ \rho_r \} - 1} \right)\\
        &= \sum_{j = 1}^\ell b_{s,j} F(x+h_j)  + O_{F,\mathcal{H}} \left(x^{\{ \rho_r \} - 1} \right),
    \end{align*}
    where in the last line we use Equation \eqref{taylor expansion} and the fact that $b_{s,j}$ depends only on $\ell,s,p_{k,j}, q_{k,j}$, which, in turn, depend only on $\mathcal{H}$.\\
    \end{proof}

As a consequence of \cref{prop_lin_comb}, assuming $U$ is of the form $(1- \varepsilon,1) \cup [0,\varepsilon)$ for some $\varepsilon > 0$, then for all $N$ large enough, we can find a subset of \eqref{good_set} that can be shown to satisfy the hypotheses of \cref{maynard} by using the results of Sections \ref{discr_integers} and \ref{discr_primes}.

\begin{Proposition}{\label{shrunkset}}
    Let $F(x) = \sum_{i=1}^r d_i x^{\rho_i}$ with $0 \leq \rho_1 < \ldots < \rho_r$ be a polynomial with real exponents, with $\rho_r \not\in \Z, d_r \neq 0$ and $\ell = \lfloor \rho_r \rfloor + 1$. Let $d \geq \ell$ be an integer and let $\epsilon > 0$ be any real number. For these $d,\ell, F(x)$, let $(b_{s,1}, \ldots, b_{s, \ell}) \in \Z^{\ell}$, $\ell + 1 \leqslant s \leqslant d$ be the vectors obtained in the conclusion of Proposition \ref{prop_relations}. Denote
$$\Delta = \frac{\varepsilon}{2\ell \max \{|b_{s,j}| : \ell < s \leq d, 1 \leq j \leq \ell \}}.$$
Then, there exists an integer $N_0$ such that
    \begin{align*}
    \mathcal{B} := \{ n \in \N : n \geq N_0 , \{ F(n+ h_j) \} \in [0, \Delta) \text{ for } 1 \leq j \leq \ell \} \\
    \subset \{ n \in \N : n \geq N_0 , \lVert F(n+ h_j) \rVert < \varepsilon \text{ for } 1 \leq j \leq d \}.
\end{align*}
\end{Proposition}
\begin{proof}
    Notice that if 
\begin{align*}
     \{F(n+h_j)\} \in [0, \Delta)
\end{align*}
for $j = 1, \ldots, \ell$, then for every $\ell < k \leq d$ and every $1 \leq j \leq \ell$,
\begin{align*}
    \lVert b_{k,j} F(n+h_j) \rVert < \frac{\varepsilon}{2\ell}.
\end{align*}
Thus, if $n$ is large enough in terms of $F$, the $h_i$'s, $d$, and $\ell$ (say $n \geq N_0$), we have
\begin{align*}
    \lVert F(n+h_k) \rVert <  \sum_{j = 1}^\ell \lVert b_{k,j} F(x+h_j) \rVert + \frac{\varepsilon}{2}< \varepsilon.
\end{align*}
It follows that
\begin{align*}
    \mathcal{B} := \{ n \in \N : n \geq N_0 , \{ F(n+ h_j) \} \in [0, \Delta) \text{ for } 1 \leq j \leq \ell \} \\
    \subset \{ n \in \N : n \geq N_0 , \lVert F(n+ h_j) \rVert < \varepsilon \text{ for } 1 \leq j \leq d \}.
\end{align*}
\end{proof}

\begin{proof}[Proof of Theorem \ref{main_result}]
    First, notice that it suffices to prove the theorem for $U$ a non-empty open set of the form $(1-\varepsilon, 1) \cup [0, \epsilon)$ for some $\varepsilon > 0$. For any other open sets, we can choose a constant $c$ such that $U + c \bmod 1$ contains a set $U'$ of the above form, and apply the theorem to $F(x) + c$ and $U'$. In this case, the relevant set $\mathcal{A}$ takes the form $\mathcal{A} := \{ n \in \N : ||F(n)|| < \varepsilon \}$. Let $d \geq \max\{\ell, 2\}$ be an integer. By \cref{shrunkset}, there exists an admissible set of polynomials $\mathcal{L} = \{ n + h_1, \ldots, n + h_d \}$ such that $h_1 < \ldots < h_d$ and with
    \begin{align*}
        \mathcal{B} := \{ n \in \N : n \geq N_0 , \{ F(n+ h_j) \} \in [0, \Delta) \text{ for } 1 \leq j \leq \ell \} \\
    \subset \{ n \in \N : n \geq N_0 , \lVert F(n+ h_j) \rVert < \varepsilon \text{ for } 1 \leq j \leq d \}
    \end{align*}
    for some $0 < \Delta < \varepsilon$ and some integer $N_0 \in \N$. We show that $\mathcal{B}$ together with the set of admissible polynomials $\mathcal{L}$ satisfy the hypotheses of \cref{maynard} for $B = \alpha = 1$ and some $0 < \theta < 1$ depending only on $F$.\\
    
    For Hypothesis \ref{H1}, let $\delta = 0.98 > 1/\log d$. Let $h$ be any element of $\{h_1, \ldots h_d\}$. Then, notice that $n \in \mathcal{B}$ if and only if $n+h$ is in 
$$
\mathcal{C} = \{ n \in \N : n \geq N_0 + h , \{ F(n - h + h_j) \} \in [0, \Delta) \text{ for } 1 \leq j \leq \ell \}.
$$
Therefore, by \cref{Hypothesis 1} applied to $h$, for all $N$ large enough in terms of $F$ and $\mathcal{L}$, we have that
\begin{align*}
    \sum_{ \substack{N \leq n < 2N \\ n+ h \in \mathbb{P}}} \mathbf{1}_{\mathcal{B}}(n) &= \sum_{ \substack{N + h \leq n < 2N + h \\ n \in \mathbb{P}}} \mathbf{1}_{\mathcal{C}}(n)\\
    &\geq \frac{0.99}{\log N} \sum_{ N \leq n < 2N } \mathbf{1}_{\mathcal{C}}(n) - h\\
    &\geq \frac{0.99}{\log N} \sum_{ N \leq n < 2N } \mathbf{1}_{\mathcal{B}}(n) - \left( 1 + \frac{0.99}{\log N} \right) h\\
    &\geq \frac{\delta}{\log N} \sum_{ N \leq n < 2N } \mathbf{1}_{\mathcal{B}}(n).
\end{align*}
In particular, we have that
\begin{align*}
    \frac{1}{d} \sum_{j = 1}^d \sum_{ \substack{N \leq n < 2N \\ n+ h_j \in \mathbb{P}}} \mathbf{1}_{\mathcal{B}}(n) \geq \frac{\delta}{\log N} \sum_{ N \leq n < 2N } \mathbf{1}_{\mathcal{B}}(n)
\end{align*}
as long as $N$ is large enough.\\

For Hypotheses \ref{H2} and $\ref{H3}$, it follows from \cref{l13} and \cref{Hypothesis 3} that there are some $0 < \theta_1, \theta_2 < 1$ depending only on $F$ such that
\begin{align*}
    \sum_{q\leq N^{\theta_1}} \max_{0 \leq c \leq q-1} \left|\sum_{\substack{N\leq n < 2N \\ n \equiv c \bmod q}}\mathbf{1}_{\mathcal{B}}(n)-\frac{1}{q} \sum_{N\leq n < 2N}\mathbf{1}_{\mathcal{B}}(n)\right| \ll_{F, d, \varepsilon} \frac{1}{(\log N)^{101d^2}}\sum_{N\leq n < 2N}\mathbf{1}_{\mathcal{B}}(n)
\end{align*}
and, for every $h \in \{ h_1, \ldots, h_d \}$,
\begin{align*}
    &
    \sum_{q<N^{\theta_2}} \max_{(q,a + h)=1}\left|\sum_{\substack{N\le n < 2N\\n\equiv a\bmod q \\ n+h \in \P}}\textbf{1}_{\mathcal{B}}(n)-\frac{1}{\varphi(q)}\sum_{\substack{N\le n< 2N \\ n + h  \in \P}}\textbf{1}_{\mathcal{B}}(n)\right|\\
    &\qquad \ll
    \sum_{q<N^{\theta_2}} \max_{(q,a)=1}\left|\sum_{\substack{N\le n < 2N\\n\equiv a\bmod q \\ n \in \P}}\textbf{1}_{\mathcal{C}}(n)-\frac{1}{\varphi(q)}\sum_{\substack{N\le n< 2N \\ n  \in \P}}\textbf{1}_{\mathcal{C}}(n)\right| + h N^{\theta_2}\\
    &\qquad\ll_{F,d,\varepsilon} \frac{1}{(\log N)^{101d^2}} \sum\limits_{\substack{N\le n< 2N \\ n \in \P}}\textbf{1}_{\mathcal{C}}(n) +  N^{\theta_2}\\
    &\qquad\ll_{F,d, \varepsilon}  \frac{1}{(\log N)^{101d^2}} \sum\limits_{\substack{N\le n< 2N \\ n + h \in \P}}\textbf{1}_{\mathcal{B}}(n),
\end{align*}
where, for the last step, we have used the fact that $\theta_2 < 1$ and the fact that
\begin{align*}
    \sum\limits_{\substack{N\le n< 2N \\ n + h \in \P}}\textbf{1}_{\mathcal{B}}(n) \gg_{F,d,\varepsilon} \pi(N,2N),
\end{align*}
which follows from the proof of \cref{Hypothesis 3}. The result obviously also holds if $\theta_1$ and $\theta_2$ are replaced by $\theta = \min \{ \theta_1, \theta_2 \}$. This $\theta$ is again dependent on $F$ only, and, in particular, is independent of the choice of $d$.\\

Finally, Hypothesis \ref{H4} is \cref{hypothesis4}.\\

It follows from the conclusion of \cref{maynard} that there exists some constant $C$ depending only on $\theta$, and, therefore, only on $F$, such that for all $d$ and $N$ large enough, the number of $n \in \mathcal{B} \cap [N,2N)$ such that the set $\{ n + h_1, \ldots, n + h_d \}$ contains a string of at least $\lceil 0.98 C^{-1} \log d \rceil$ consecutive primes is
\begin{align*}
    \gg  \frac{\sum_{N \leq n < 2N} 1_{\mathcal{B} }(n)}{(\log N)^d \exp(Cd^5)},
\end{align*}
where the implied constant may depend on $F$ and $\varepsilon$, but not on $N$ and $d$. It follows from the proof of \cref{l13} that
\begin{align*}
    \sum_{N \leq n < 2N} 1_{\mathcal{B} } (n)\gg_{F, d, \varepsilon} N
\end{align*}
for all $N$ large enough, so, for any fixed $d$, for all $N$ large enough in terms of $F$, $d$ and $\varepsilon$, there will be at least one $n \in \mathcal{B} \cap [N,2N)$ such that the set $\{ n + h_1, \ldots, n + h_d \}$ contains a string of at least $\lceil 0.98 C^{-1} \log d \rceil$ consecutive primes, and, by our choice of $\mathcal{B}$, they are all contained in the set $\mathcal{A}$. Therefore, for any $m \in \N$, by taking $d$ large enough (so that all the conditions in \cref{maynard} are satisfied) and with $\lceil 0.98 C^{-1} \log d \rceil \geq m$, there are infinitely many strings of $m$ consecutive primes of bounded length (depending on $F$ and $m$) inside $\mathcal{A}$.

\end{proof}



\bibliographystyle{alpha}
\bibliography{biblio}

\bigskip
\footnotesize
\noindent
Sai Sanjeev Balakrishnan\\
\textsc{Department of Mathematics, \\ University of California, 
 Berkeley, \\ Evans Hall, Berkeley, California, \\ United States of America - 94720-3840}\\
\href{mailto:ssb55@cam.ac.uk}
{\texttt{saisanjeev@berkeley.edu}}
\\ \\
Félix Houde\\
\textsc{Department of Mathematics and Statistics,}\\
\textsc{Concordia University,}\\
\textsc{1455 de Maisonneuve Blvd West, Montreal, Quebec, Canada, H3G 1M8}\\
\href{mailto:f_houd@live.concordia.ca}
{\texttt{f\_houd@live.concordia.ca}}
\\ \\
Vahagn Hovhannisyan\\
\textsc{Faculty of Mathematics and Mechanics,}\\
\textsc{Yerevan State University,}\\
\textsc{1 Alek Manukyan St, Yerevan, Armenia,}\\
\href{mailto:vahagn.hovhannisyan3@edu.ysu.am}
{\texttt{vahagn.hovhannisyan3@edu.ysu.am}}
\\ \\
Maryna Manskova\\
\textsc{Institute of Analysis and Number Theory,\\}
\textsc{Graz University of Technology,}\\
\textsc{Steyrergasse 30, 8010 Graz, Austria,}\\
\href{mailto:maryna.manskova@tugraz.at}
{\texttt{maryna.manskova@tugraz.at}}
\\ \\
Yiqing Wang \\
\textsc{Department of Mathematics,}\\
\textsc{University of Wisconsin–Madison,}\\
\textsc{480 Lincoln Dr, Madison, WI 53706, USA}\\
\href{mailto:wang3498@wisc.edu}{\texttt{wang3498@wisc.edu}} \\ 


\end{document}